\documentclass[11pt]{article}
\usepackage{amsmath, amsthm, amssymb, setspace, mathrsfs, mathtools, mathpazo, enumitem, graphicx, xypic}
\makeatletter
\renewcommand\section{\@startsection{section}{1}{\z@}%
 						{-3.5ex \@plus -1ex \@minus -.2ex}% negative means
						{2ex \@plus.2ex}% 		positive means vertical skip
						{\large\bfseries}}
\renewcommand\subsection{\@ifstar
						{\setcounter{subsection}{\value{equation}}
					\@startsection{subsection}{2}{\z@}
                          {1.75ex \@plus.5ex \@minus.2ex}%
                           {-.4em}		% negative means horizontal
					\textit*}
					{\setcounter{subsection}{\value{equation}}
						\stepcounter{equation}
					\@startsection{subsection}{2}{\z@}
                          {1.75ex \@plus.5ex \@minus.2ex}%
                           {-.4em}		% negative means horizontal
					\textit}}
\def\@seccntformat#1{\@ifundefined{#1@cntformat}%
	{\csname the#1\endcsname\quad} 
	{\csname #1@cntformat\endcsname}} 
\def\section@cntformat{\thesection.~} 
\def\subsection@cntformat{(\thesubsection)\ }
\renewcommand*\l@section{\mdseries\small\@dottedtocline{1}{1.5em}{2em}}
\makeatother

\DeclareMathOperator{\coker}{coker}

\numberwithin{equation}{section}
\theoremstyle{plain}
\newtheorem{maintheorem}{Theorem}
\swapnumbers
\newtheorem{theorem}[equation]{Theorem}
\newtheorem{corollary}[equation]{Corollary}

\newtheorem{proposition}[equation]{Proposition}

\theoremstyle{definition}
\newtheorem{definition}[equation]{Definition}
\newtheorem{remark}[equation]{Remark}

\newcommand{\script}{\mathscr}
\newcommand{\mrE}{\mathrm{E}}
\newcommand{\cliff}{\mathrm{Cliff}}

\newcommand{\cG}{\script{G}}

\newcommand{\cO}{\script{O}}

\newcommand{\cT}{\script{T}}

\newcommand{\bC}{\mathbb{C}}

\newcommand{\bI}{\mathbb{I}}

\newcommand{\bR}{\mathbb{R}}

\newcommand{\bS}{\mathbb{S}}
\newcommand{\bT}{\mathbb{T}}
\newcommand{\bZ}{\mathbb{Z}}
\newcommand{\vep}{\varepsilon}

\newcommand{\bmu}{\boldsymbol{\mu}}
\newcommand{\Spin}{\mathrm{Spin}}
\newcommand{\SO}{\mathrm{SO}}
\newcommand{\Serre}{S}
\newcommand{\vect}{\mathbb{V}}
\newcommand{\braid}{\mathbb{B}}
\newcommand{\fus}{\mathbb{F}}
\newcommand{\ferm}{\psi}
\newcommand{\cats}{\mathbb{L}}
\newcommand{\spit}{\nu}
										
\begin{document}
\title{\textbf{Fully local Reshetikhin--Turaev theories}}
\author{Daniel S.~Freed, Claudia I.~Scheimbauer, Constantin Teleman}
\date{\today}
\maketitle
%%%%%%%%%%%%%%%%%%%%%%%%% ABSTRACT %%%%%%%%%%%%%%%%%%%%%%%%%%%%%%%%%%%%%%%%%
\begin{quote}
\abstract{
\noindent 
Completing an arc of research initiated by Reshetikhin--Turaev and Witten, we construct \emph
{fully local} 3-dimensional topological field theories from non-degenerate braided fusion 
categories: we define a symmetric tensor enhancement~$\mrE\fus$  with full 
duals of the $3$-category~$\fus$ of fusion categories in which every Reshetikhin--Turaev theory 
has a point generator. This  
$\mrE\fus$ is a direct sum of invertible $\fus$-modules,  indexed by a $\bmu_6$-extension of the 
Witt group of non-degenerate braided fusion categories. Similarly, we enhance the $3$-category 
$S\fus$ of fusion super-categories to a symmetric tensor $3$-category $\mrE S\fus$ with full duals, 
which is a sum of invertible $S\fus$-modules indexed by a $\bmu_{24}$-extension of the super-Witt 
group. The unit spectrum of $\mrE S\fus$ is the 
connective cover of the Pontrjagin dual of $\bS^{-3}$. We discuss  tangential structures and central 
charges of the resulting TQFTs. We establish the $\mathrm{Spin}$-invariance of fusion 
supercategories, and relate $\SO$-invariance  to modular and spherical structures, 
confirming some conjectures from \cite{dsps}.  
} 
\end{quote}

%%%%%%%%%%%%%%%%%%%%%%%%%%%%%%%%%%%%%%%%%%%%%%%%%%%%%%%%%%%%%%%%%%%%%%%%%%%%% 

%%%%%%%%%%%%%%%%%%%%% Table of Contents %%%%%%%%%%%%%%%%%%%%%%%%%
\section*{}
\begin{center}
\begin{minipage}[t]{12cm}
\tableofcontents
\end{minipage}
\end{center}
%%%%%%%%%%%%%%%%%%%%%%%%%%%%%%%%%%%%%%%%%%%%%%%%%%%%%%%%%%%%%%%%

\section*{Introduction}
\addcontentsline{toc}{section}{\protect\numberline{}Introduction}
\subsection{Background.} 
The $3$-dimensional topological quantum field theories (TQFTs) constructed by Reshetikhin--Turaev 
\cite{rt} and Witten \cite{w} are (symmetric monoidal) functors from a bordism $2$-category of compact 
manifolds of dimensions $1,2,3$ to the target symmetric tensor $2$-category~$\cats$ with objects 
finite, semi-simple $\bC$-linear categories (enriched over the symmetric monoidal category $\vect$ 
of complex vector spaces), and with functors and natural transformations as $1$-and 
$2$-morphisms. On oriented manifolds, the theories are only \emph{projectively} defined. 
The projective dependence is controlled by a \emph{central charge~$c\in \bC$},  normalized to align 
with the Virasoro central charge of (rational) conformal field theories, whose chiral sectors live 
on the boundary of these TQFTs. Thus, the 
chiral free fermion conformal theory has $c=1/2$, and lives on the boundary of the \emph{topological 
free fermion} theory~$\ferm$, an invertible, albeit $\Spin$-dependent TQFT. Restricted to framed 
manifolds, $\ferm$ defines the standard character $\pi_3^s \to \bmu_{24}$ of the framed bordism group. 

The scalar anomaly in the TQFT gluing laws can be resolved by choosing a suitable \emph 
{$3$-dimensional tangential structure $\tau$}, and supplying all manifolds with a $\tau$-structure 
on a $3$-dimensional germ. Certain shifts in the structure scale the output by $c$-dependent, top-level 
automorphisms; this also renders $c$ ambiguous by a ($\tau$-dependent) additive shift. 
This scaling is easy to explain: a change of structure that can be implemented locally, near a single 
point of the manifold, affects the theory by point operators, which in these theories are scalar. 
By contrast,  changes in $\Spin$ structure are not localized at single points, and  $\Spin$-dependent 
versions of Reshetikhin--Turaev (RT) theories show a more complicated behavior~\cite{bel, j}.

Several tangential structures may be used. The default is a \emph{$3$-framing}; a shift in framing 
by $1\in\pi_3\,\SO_3 = \bZ$, which can be effected near a single point, transforms the invariant 
of a closed, framed~$3$-manifold by a factor\footnote{Some early literature normalized~$c$ incorrectly, 
by a factor of $2$.} of~$\exp\left(2\pi i c/6\right)$. Less restrictive is an \emph
{$\SO^{p_1}$-structure} (oriented $p_1$-structure, also called a $(w_1,p_1)$-structure), a trivialization 
of the Pontrjagin cocycle~$p_1$ on oriented manifolds; this was first considered in \cite{bhmv}. A unit 
shift in $p_1$-trivialization changes the top-level invariants by $\exp\left(2\pi i c/24\right)$, 
reflecting the value~$p_1=4$ on the basic spherical class in~$B\SO_3$. Mediating 
between these are $\Spin^{r p_1}$-structures  ($r=1, 1/2, 1/4$), incorporating trivializations 
of $r p_1$ on $\Spin$ manifolds; we  review those in \S\ref{spinomaly},
with more details in~Appendix~\ref{tangents}. One novelty introduced in \S\ref{Qp1} 
is a \emph{$\bC p_1$-structure}, which lifts the central charge to~$\bC$. The role of 
$p_1$-structures in Chern-Simons theories,  following Witten's construction \cite{w}, is 
reviewed in \cite{ft2}.

\subsection{RT theories from categories.} 
An anomalous RT theory $\cT$ is determined by the data of
\begin{enumerate}\itemsep0ex
\item a \emph{finite, semi-simple} category $T:=\cT(S^1)$, associated to the standard circle, 
\item a \emph{rigid, braided tensor structure} on $T$, defined from the pair-of-pants multiplication, 
\item  a \emph{ribbon} automorphism of the identity of $T$, defined from the circle rotation 
action,
\end{enumerate}
which are subject to two conditions: 
\begin{enumerate}\itemsep0ex
\item[(a)]  \emph{non-degeneracy} of the braiding: the central objects form the additive summand $\vect\cdot\mathbf{1}\subset T$, 
\item[(b)] \emph{symmetry} of the ribbon, giving a homogeneous quadratic enhancement of the braiding. 
\end{enumerate}
A braided fusion category (BFC), defined by (i) and (ii), is called \emph{non-degenerate} if it satisfies 
condition (a). The complete package (i--iii) with (a,b) describes a \emph{modular tensor category} (MTC). 
Viewing~$T$ as an object in a $4$-category\footnote{There are several variant $4$-categories; 
for this statement, $E_2$ objects in linear categories will do, see e.g.~\cite{bjs}.}  of BFCs, the 
ribbon trivializes the square of the braiding,\footnote {Equivalently: in the pair of 
pants, the square of the braiding is the product of the three boundary Dehn twists.} which is the 
structural \emph{Serre automorphism}: see~\S\ref{orient}. Symmetry ensures that the trivialization 
has order $2$. These data and constraints very nearly determine a linear TQFT~$\cT$ for oriented 
manifolds with \emph{signature structure} in dimensions $1/2/3$: see for instance Turaev's book~\cite{tur}, 
or the detailed account in~\cite{bdsv}. Left over is a sign ambiguity 
(see \cite{bk, bdsv} and \S\ref {sabotage} below). The central charges of the two resulting theories 
couple to $1/8^{th}$ of the signature. 

As we shall review, $\Spin$ versions of RT theories can be defined from braided \emph{super-categories} --- 
linear over super-vector spaces --- meeting a non-degeneracy condition akin to (a): the central objects 
form a copy of $S\vect \subset T$. There will be $24$ choices of framed theories for a given~$T$, 
differing by powers of the topological free Fermion theory~$\ferm$ (see~\S\ref{alphaomega}). Other 
tangential choices lead to  different sets of options; for instance, $\Spin^{p_1}$-structures 
give $48\times 2$ choices.

\subsection{Formulation of the problem.}
Although `one step more local' than Atiyah-Segal TQFTs \cite{at} in including corners of 
co-dimension $2$, Reshetikhin--Turaev theories are not manifestly \emph{fully local} in the sense of the 
Baez-Dolan-Lurie \emph{Cobordism Hypothesis}~\cite{l}: they are not generated by an object $X = \cT(pt)$ 
in an obvious symmetric tensor $3$-category. For such an $X$, the endomorphism category 
of $\mathrm{Id}_X$ --- the \emph{Drinfeld center $Z(X)$} --- should be braided-tensor isomorphic to~$T$. 
The obstruction to finding such a \emph{fusion category $X$} is the class~$[T]$ in the \emph
{Witt group $W$} of non-degenerate braided fusion categories~\cite {mug}, the quotient of 
invertible~BFCs by Drinfeld centers. If $[T]=0$, then such an~$X$ exists, uniquely up to isomorphism 
in~$\fus$; the resulting fully local theory is then a \emph{Turaev-Viro TQFT}. 
Absent such a fusion category $X$, the theory $\cT$ is more difficult to access.

The question of fully localizing RT theories, especially in relation to Chern-Simons theories, 
was raised from the early days of the Cobordism Hypothesis, e.g.~\cite{fhlt}. Constructions have been 
proposed, based on conformal boundary theories, in terms of vertex algebras or nets of von Neumann 
algebras~\cite{k, andre}. A feature of these approaches is their use of \emph{additional analytic input}, 
and the apparent \emph{absence of additional topological output}. This leads to the suspicion that the 
information required to fully localize these TQFTs is entirely contained in their $1/2/3$ portion. 

\subsection{The universal category.}\label{results}
In this paper, we confirm this suspicion, and enlarge the $3$-category $\fus$ of fusion categories 
to a \emph{symmetric monoidal $3$-category $\mrE\fus$} (``enlarged $\fus$''),  
containing the point generators of Reshetikhin-Turaev theories. It is characterized by the  
universality property~(viii) below, and has the additional properties (i-vii):
\begin{enumerate}\itemsep0ex
\item $\mrE\fus$ \emph{has full duals}: all $k$-morphisms are $(3-k)$-dualizable, $0\le k \le 3$;
\item $\mrE\fus = \bigoplus_{\widetilde{w}\in \widetilde W} \fus_{\widetilde{w}}$, where $\widetilde{W} \to W$ is an extension 
of the Witt group by $\bmu_6$;
\item In particular, there are no non-zero morphisms relating objects in $\fus_{\widetilde{w}}$ and 
$\fus_{\widetilde{w}'}$ when $\widetilde{w}\neq \widetilde{w}'$;
\item Each $\fus_{\widetilde{w}}$ is an invertible module over $\fus_1 = \fus$; 
\item Specifically, choosing a representative braided category~$T(w)$ for $w\in W$ gives 
an isomorphism $\fus_{\widetilde{w}} \cong \fus_{T(w)}$, the $3$-category of fusion categories with central 
action of $T(w)$ (called \emph{fusion categories over $T(w)$});
\item  When $\zeta=\exp(2k\pi i/6)\in \bmu_6\subset\widetilde{W}$, we have $\fus_\zeta \cong \fus$ as a module, generated by an invertible object $U^{\otimes k}$, unique up to isomorphism; 
\item The units $U^{\otimes k}$ generate the six invertible framed TQFTs valued in $\mrE\fus$, and 
factor uniquely through 
the category of oriented manifolds with $p_1$-structure;
\item Every\footnote{Subject to some standard assumptions, see Theorem~1.} symmetric monoidal 
$3$-category~$\bT$ which contains~$\fus$, and for which each framed $1/2/3$ TQFT valued in 
$\cats$ extends uniquely to a fully local $\bT$-valued theory,  receives a unique 
symmetric tensor functor from $\mrE\fus$.
\end{enumerate} 

\subsection{Commentary.}
Tempting as it is to assert that $\mrE\fus$ is universal in promoting Reshetikhin--Turaev 
theories to fully local, framed theories,  this needs qualification: it is not known 
if every non-degenerate BFC admits modular structures. 
A category which is universal in that weaker sense  only receives a functor out of the 
`modular part' of  $\mrE\fus$. 

Property~(iii) seems disappointing, but follows from the main result of~\cite{ft1}: 
no non-zero topological interfaces exist between Witt inequivalent, fully local $3$-dimensional 
theories. This  counters the traditional supposition that the point generator for a $3$D 
TQFT must be the $2$-category of its topological boundary theories:  
none such exist for objects in $\mrE\fus\setminus\fus_1$. Property~(vii) is specific to 
\emph{bosonic} theories, with no super-vector spaces in the target. As a consequence, the 
subgroup $\bZ/4\subset\pi_3^s \cong \bZ/24$ is represented \emph{trivially} in a bosonic, 
invertible~$3$D theory, and~$\pi_3^s$ must factor through the quotient group~$\bZ/6\cong 
\pi_3 (\bS/\langle{\eta}\rangle)$ (Remark~\ref {Smodeta}). This agrees with the bordism group 
of $p_1$-oriented $3$-manifolds (Appendix~\ref{tangents}). 

\subsection{The super-category.}
The same method  gives a super-version of the result, relevant for Spin RT theories, considered 
early on in \cite{bm, bel}, or more recently \cite{j}: we get an enlargement $\mrE S\fus$ of 
the $3$-category $S\fus$ of fusion super-categories by the super-Witt group~$SW$ of non-degenerate 
BFSCs modulo Drinfeld centers from~$S\fus$ \cite{don}. The kernel of the extension $\widetilde{SW}\to SW$ 
is now the Pontrjagin dual of $\pi_3^s = \bZ/24$, and the objects of $\mrE S\fus$ generate the 
framed $3$-dimensional TQFTs.

\subsection{Orientations and spherical structures.} 
We also address the question of relaxing the domain of these (super) 
TQFTs. \emph{A priori}, they are defined on $3$--framed manifolds, and their $p_1$-dependence precludes a clean 
factorization through $\Spin$ or oriented manifolds. However, as we show in 
\S\ref{spinv} and \S\ref{canorient}, the topological boundary theories for Turaev-Viro (super) TQFTs 
constrain this $p_1$-dependence, force the vanishing of (appropriately reduced) central charges,  
and descend them to~$\Spin$ or, at times, oriented TQFTs. This was conjectured in \cite{dsps}.

In particular, we spell out the relation (partially established in \cite{tur}) between $\SO^{p_1}_3$-invariance 
data of objects $X\in \mrE\fus$ and modular structures on their centers. We find that not all modular 
structures are created equal: if they exist, there is a preferred one\footnote{One may speculate 
that the preferred modular structure is unitarizable, but this is not clear from our construction.}  
(Theorem~\ref{canonical}). The others correspond 
to invertible central elements of order~$2$.

When~$X=F$ is a fusion category, the preferred modular structure leads to $\SO_3$-invariance data on~$\cT_F$, 
and allows \emph{compatible $\SO_2$-invariance data on all its boundary theories}. Changing the modular 
structure by~$z\in Z(F)$ of order~$2$ obstructs the $\SO_3$-invariance precisely when~$z$ is 
\emph{fermionic} (has self-braiding $-1$): see the example of $\bZ/2$-gauge theory in \S\ref{funnyZ2}. 
At the other extreme, central lifts of order~$2$ of $\mathbf{1}\in F$ also preserve the invariance of 
the regular module. More precisely, spherical structures on~$F$ correspond to \emph{$\SO_3$-invariance 
data for~$F$ and compatible $\SO_2$-invariance of its regular module}\footnote{Modulo the 
choice of some non-zero scalars: a \emph{Frobenius structure constant} for each summand of $F$.}  
(Theorem~\ref{spherical}). This refines a conjecture of~\cite{dsps}, and meshes well with the 
main theorem of \cite{ft1}, which characterizes (simple) fusion categories as $3$-dimensional 
simple TQFTs equipped with a non-zero 
boundary theory (the regular module).

\subsection{Central charges.} \label{sabotage}
The $4$-fold factor between framing- and $p_1$-shifts is broken for more general tangential structures. 
The problem stems from the co-kernel~$\bZ/2$ of $J:\pi_3\SO_3\to \pi_3^s$ (whereas~$J$ surjects from 
$\pi_3\SO$.) One TQFT manifestation is the invertible theory~$\omega$ with central charge~$6$ (see \S\ref
{alphaomega}), which is trivial on $3$-framed manifolds, but not so on  $\Spin_3^{rp_1}$-structured 
manifolds. Another  is the invertible $3$-framed theory~$\spit$ of 
order two, which extends to a 
$\Spin$-theory with central charge~$0$, as well as an oriented $p_1$-version with central 
charge~$12\bmod{24}$. The $\Spin$ theory~$\spit$ detects the unstable~$\bZ/2$ summands 
in several $3$-dimensional bordism groups (Appendix~\ref {tangents}, Table~\ref{mytable}).

A (logarithm of a) Gauss sum, defined from $T$ and its ribbon, determines a rational number 
$\bmod{\:8\bZ}$ \cite{mug2, rowell}. Our  saboteur, $\coker J$, amends the folk belief that identifies 
this number with~$c(\cT)\bmod{8\bZ}$. As already noted in \cite{bk, bdsv}, there are \emph{two} 
signature-structured $1/2/3$-dimensional TQFTs for oriented manifolds defined by the same modular~$T$: 
they differ by the invertible theory~$\spit$, which has central charge~$4\bmod{8\bZ}$ when promoted 
to oriented manifolds with signature structure.\footnote{Because ~$\spit$ is not \emph{reflection-positive} 
in the classification of \cite{fh}, it is not clear 
if this central charge ambiguity 
can affect \emph{unitary} TQFTs, and we do not discuss unitarity here.}

We prefer the setting of $p_1$-structures, which genuinely localize to points: signature structures 
only do so projectively, \cite{projsym}. Then, a projective $1/2/3$-theory as above has~$6$ 
fully local linearizations on $\SO^{p_1}$-manifolds, related by powers of the invertible 
theory~$\cT_U$. Their central charges, defined $\bmod{24}$, agree with the Gauss sum~$\bmod{\:4}$. From this perspective, signature structures are an (unsuccessful) 
attempt to remove the linearization ambiguity, pinning the even powers of~$\cT_U$. The sign ambiguity remains. 

\subsection{Central extensions of $\mrE S\fus$.} Objects of $\mrE S\fus$ classify $3$-dimensional 
framed TQFTs, and we may ask about alternative tangential structures. For 
$\Spin^{rp_1}_3$-structures ($r\in \frac{1}{4}\bZ$ or $r=\bC$, see \S\ref{Qp1}), this is closely 
related to the question of refining reduced central charges, and we have a ready answer: the summands 
of the classifying $3$-category are now labelled by a central extension of~$SW$ by the Pontrjagin 
dual of $\pi_0MT\Spin_3^{rp_1}$ (Appendix~\ref{tangents}, 
Table~\ref{mytable}). Indeed, thanks to the Cobordism Hypothesis, we seek the $\Spin_3^{rp_1}$-fixed 
point category of~$\mrE S\fus$. As we show in \S\ref{spinv}, the action of $\Spin_3$ depends 
only on the super-Witt label $\widetilde{w}$ of a summand $S\fus_{\widetilde{w}}$, and the action 
of $\Spin_3^{rp_1}$ becomes trivial, since $\pi_3\Spin_3$ has been killed. This leads to an extension of 
$\widetilde{SW}$ by $\bmu_{4r}$ (or by $\bZ$, when $r=\bC$). The center of the total 
extension of~$SW$ is the group of invertible theories, dual to the bordism group. The central charge now 
picks up the stable summand, and refines accordingly.

\subsection{Four dimensions.} 
If we invoke the recent classification of fusion $2$-categories \cite{f2cat}, we can continue our 
discussion in one higher dimension. The new objects $X\in\mrE\fus$ become new \emph
{$1$-morphisms} in the symmetric tensor $4$-category of fusion $2$-categories:  specifically, 
isomorphisms of their centers~$Z(X)$ with the unit category $\vect$ (as a BFC). This will 
kill the (super-)Witt group, and  produce a fully dualizable symmetric monoidal $4$-category, wherein  
all TQFTs become finite gauge theories with 
generalized Dijkgraaf-Witten twists. Conjecturally~\cite{kw}, these are \emph{all} of the fully local, 
$4$-dimensional TQFTs valued in $4$-categories (as opposed to $(\infty,4)$-categories), so the 
output has the flavor of a universal target for $4$-dimensional TQFTs. However, the new objects 
in~$\mrE \fus$ give rise to a larger $4$-category: see for instance the novel $\bZ/3$-gauge theory in 
Appendix~\ref{exgauge}, which incorporates a Dijkgraaf-Witten twist from $\pi_3^s$ that is not accessible 
in~\cite{f2cat}. The genuine universal target $4$-category thus still needs 
spelling out. A model characterization, and for its generalization to higher dimensions, 
has  been announced by Johnson-Freyd and Reutter. We hope to return to $4$~dimensions in the future.

\subsection{Chromatic Kummer theory.} It was proposed some time ago by Mike Hopkins that the Pontrjagin dual 
spectrum $\bI_{\bC^\times}$ of the sphere should form the units of a good universal target for homotopy theory. 
The top groups $\bZ/2$ were long recognized as classifying invertible graded lines and graded 
invertible algebras, but the meaning of the next group $\bZ/24$ remained mysterious. The recent 
development~\cite{csy} of chromatic Kummer theory led to suggestions that an \emph{cyclotomic 
closure condition} should be invoked to explain the emergence of $\bI_{\bC^\times}$. While our work 
confirms Hopkins' conjecture through dimension~$3$ (and very nearly~$4$) by bringing in the $\bZ/24$ 
group and  killing the next (super-Witt) group, it \emph{does not align so well} with the speculative 
cyclotomic-closure motivation. Instead, the Pontrjagin dual~$\bI_{\bC^\times}$ appears exactly 
as it should, in classifying symmetry 
structures in the target category. The relation to Kummer 
theory is a consequence, not an input.

\subsection{Related prior work.} Our theorems rely on the dualizability properties of fusion and 
braided fusion categories, established in \cite{dsps, bjs, bjss}. The other key input is Kevin 
Walker's~\cite{Wa} presentation of Reshetikhin-Turaev theories as \emph{fully local anomalous} TQFTs, 
living on the boundary of $4$-dimensional Crane-Yetter theories;  see \cite{haioun} for a modern exposition.  
Specifically, a modular tensor category $T$ generates a fully local, 
\emph{invertible} $4$-dimensional TQFT of oriented manifolds.  Invertibility follows from that of~$T$ 
in the higher Morita category of BFCs, a consequence of \emph{factorizability}~\cite{mug, dgno}. 
The \emph{regular module}, the fusion category~$T$ over~$T$, provides a fully local boundary 
theory. This anomalous presentation of RT theories leads to their description as linear theories in 
dimensions~$1/2/3$: the vanishing of the respective oriented bordism groups allows one to trade 
the anomaly for a tangential structure. We extend this to dimension~$0$.

\begin{remark}
A promotion of Walker's anomalous presentation to a fully local RT theory was attempted in~\cite{h}. 
However, that construction simply reformulates the anomalous RT theory, losing the manifold invariant: 
at top level, it gives a line (the Crane-Yetter line) and a vector therein. 
Extracting the numerical invariant requires a trivialization of the Crane-Yetter line. Walker's 
formulation reduces the problem of point-localizing RT theories \emph{precisely} to the \emph
{fully local} trivialization of Crane-Yetter theories. This was missed in~\cite{h}, and  
we address it in our paper. 
\end{remark}

\subsection{Extended summary.} Here is a guide to the subsequent sections of the paper.
\begin{enumerate}\itemsep0ex
\item[1.] Section~\ref{boson} constructs the bosonic enlarged category $\mrE\fus$, explaining 
the role of the symmetric group.
\item[2.] Section~\ref{fermion} repeats the construction for fusion super-categories, leading to $\mrE S\fus$. 
\item[3.] Section~\ref{spinomaly} reviews the projective $\Spin$-invariance of our TQFTs (Theorem~\ref{thm3}) 
and the topological central charge $\mu =\exp(2\pi i c/6)$. Tangential structures related to $p_1$ on $\Spin$ 
manifolds are related to anomalies (Theorem~\ref{thm4}). 
\item[4.] Section~\ref{spinv} proves the $\Spin_3$-invariance of TQFTs defined from super-categories, 
and of interfaces between them, confirming a conjecture from~\cite {dsps}. 
\item[5.] Section~\ref{orient} reviews the removal of Spin dependence in the presence of modularity. 
\item[6.] Section~\ref{canorient} classifies $p_1$-twisted orientations, introducing 
the \emph{canonical} orientation. 
\item[7.] Section~\ref{Qp1} introduces \emph{complex} $p_1$-structures, enabling the lift of 
the central charge to $\bC$. We verify its match with the central charge of boundary CFTs.   
\item[A.] This Appendix shows how the new objects in $\mrE\fus$ lead to novel $4$-dimensional TQFTs,  
with the example of a $3$-dimensional gauge theory.
\item[B.] This Appendix reviews some relevant bordism groups and maps between them.  
\end{enumerate}

\subsection*{Acknowledgements.} We had many fruitful conversations on this topic with Mike Hopkins, 
who taught us the role of the Pontrjagin dual $\bI_{\bC^\times}$ of the sphere. Over the years, we 
had useful exchanges with Chris Douglas and Andr\'e Henriques on the topic 
of Chern-Simons theory. Although our project predates those, we also benefitted from lectures 
Theo Johnson-Freyd and David Reutter on universal targets for TQFTs. Noah Snyder gave us patient accounts of 
the issues surrounding the invariance conjectures from \cite{dsps}. Repeated questions by several colleagues 
(Alexei Kitaev, Milo Moses, Nivedita, Nikita Sopenko) led us to lift the central charge to~$\bC$ by means of 
$\bC p_1$-tangential structures in \S\ref{Qp1}. 
The third author presented early versions of this work on several 
occasions, most recently at the 2022 Annual Meeting of our Simons Collaboration.

All authors were funded by the Simons Collaboration on Global Categorical Symmetries. DSF was partially 
supported by National Science Foundation grant DMS-2005286. This work was performed in part at Aspen 
Center for Physics, which is supported by NSF grant PHY-2210452.  We also thank the Harvard Center 
of Mathematical Sciences and Applications for its support. CS acknowledges funding by the Deutsche 
Forschungsgemeinschaft (DFG, German Research Foundation) through the Collaborative Research Centers SFB 
1085: Higher invariants--224262486 and SFB: 1624 Higher structures, moduli spaces and integrability--506632645.

\subsection*{Notation and abbreviations.} 
\begin{itemize}\itemsep0ex
\item TQFTs = topological quantum field theories
\item BF(S)Cs = braided fusion (super-)categories
\item $\vect$ = $\bC$-linear category of finite-dimensional vector spaces
\item $\vect^{\otimes}, \vect^{br}$ denotes the same with its tensor/braided structure
\item $\vect^{br,\zeta}_{\bZ/2}$ the BFC of $\bZ/2$-graded vector spaces with one of the 
four braided structures, labelled by $\zeta\in \{\pm1, \pm i\}$ 
\item $\cats$ =(symmetric tensor) linear $2$-category of finite semi-simple $\bC$-linear categories 
\item $\fus$ = (symmetric tensor) $3$-category of: fusion categories, finite semi-simple 
bimodule categories, functors and natural transformations
\item $\braid$ = (symmetric tensor) $4$-category of BFCs: a full subcategory of algebras in $\fus$, with 
bi-modules in $\fus$ as $1$-morphisms  and compatible 
higher morphisms\footnote{See \cite{jfs} for an iterative construction of higher categories 
of algebras.} 
\item $W$ = Witt group of invertible BFCs modulo centers
\item Prefixed $S$ means `super' (including super-vector spaces, Clifford algebras, etc.): \\
$S\vect$, $S\cats$, $S\vect^\otimes$, $S\fus$, $SW$, $S\braid$ 
(cf.~also \S\ref{fermion})
\item $\bS$ = sphere spectrum, $HA$ the Eilenberg-MacLane spectrum for an Abelian group $A$
\item $\bI_{\bC^\times}$ = Pontrjagin dual of $\bS$
\item $\cT_X$ denotes the TQFT defined by an object $X\in\mrE S\fus$ (framed, unless otherwise stated)
\item $\mathrm{GL}_1 T$ denotes the (higher) group of invertibles in a tensor (higher) category $T$
\item $\iota_{>k} C$ denotes the (higher) subcategory of $C$ keeping only invertible $j$-morphisms for $j>k$
\item $[x]$ denotes the isomorphism class of an object $x$ in the ambient (higher) category
\item For a characteristic class $\vep$ of $B\Spin_n, B\SO_n$, denote by $\Spin^{\vep}_n, \SO^{\vep}_n$ 
the groups classified by the homotopy fibers of~$\vep$ (as maps to the respective Eilenberg-MacLane space). 
\end{itemize}

%%%%%%%%%%%%%%%%%%%%%%%% SECTION %%%%%%%%%%%%%%%%%%%%

\section{Main result: bosonic case} 
\label{boson}

We state and prove the main theorem for \emph{bosonic} TQFTs --- those not including super-vector 
spaces. The construction of $\mrE\fus$ is performed in the proof below. Follow-up remarks  
elaborate on additional properties of our construction.
 
\begin{maintheorem}
There exists symmetric monoidal $3$-category $\mrE\fus \equiv \bigoplus_{\widetilde W} \fus_{\widetilde{w}}$ 
satisfying properties (i)-(vii) in \S\ref{results}. Moreover, it 
is initial among symmetric tensor $3$-categories~$\bT$ which
\begin{enumerate}\itemsep0ex
\item admit finite colimits, compatible with the tensor structure, 
\item receive a symmetric tensor functor~$\Phi$ from~$\fus$, preserving finite colimits 
\item contain a unique (up to isomorphism) point generator of every $\cats$-valued $1/2/3$-dimensional TQFT.
\end{enumerate} 
\end{maintheorem}
\noindent
The units in $\mrE\fus$ form a cyclic group of order $6$; they are characterized by their behavior 
under symmetry (Remark~\ref{projsym}).  Their reduced central charges are in $2\bZ \pmod {6\bZ}$, 
but lift to $4\bZ \pmod{24\bZ}$ under refinement of structure (see \S\ref{spinomaly}). 
Here is an amusing consequence:
\begin{corollary}
The framed  $3$-manifold invariants seen by \emph{invertible} bosonic theories are the 
characters of $\pi_3^s \cong \bZ/24$ which factor through $\bZ/6$. 
\end{corollary}
\noindent
We denote by~$U\in \mrE\fus$ the object defining the fourth power of the standard character 
of~$\pi_3^s$. By contrast, the fermionic theories in the next section can represent the full $\bZ/24$. 

\medskip
We construct $\mrE\fus$ by converting Walker's relative (anomalous) RT theories to absolute 
ones. Trivializing the~$4D$ anomaly theory defined by modular $T$ requires an object $X\in\mrE\fus$, 
together with a Witt equivalence $T \equiv Z(X):= \mathrm{End}(\mathrm{Id}_X)$. If $[T]\neq0$ 
in the Witt group, $X$ cannot be a fusion category. Instead, we add the missing pre-centers 
as direct summands to $\fus$, as follows. Recall that for a fusion category $F$, $Z(F)\equiv 
F\boxtimes F^\vee$ by a Morita equivalence in $\fus$ which matches the second algebra structures 
on the two sides (the braiding with the contraction). The imagined generators~$\{X\}$ 
of a collection of RT theories, whose centers $\{T\}$ represent the non-trivial Witt
classes, will now be portrayed by their centers $T$ (viewed as fusion categories over $T$), 
whereupon we rig the tensor product so that $T\boxtimes T^{rev}$ becomes isomorphic to $T$, 
to match what $X\boxtimes X^\vee$ should produce in~$\mrE\fus$. We must of course be coherent 
with respect to the symmetric tensor structure.     

\begin{proof}
For each $w\in W$, choose a representing non-degenerate braided fusion category $T(w)$, denote by 
$\fus_w$ the category of fusion categories over $T(w)$, and define (provisionally)
\[
E'\fus := \bigoplus\nolimits_{w\in W} \fus_{w} 
\]
with only zero morphisms between distinct summands. Each summand is an invertible $\fus$-module, 
because of the invertibility of the categories $T(w)$ in $\braid$. 

An obvious attempt at placing a symmetric tensor structure on $E'\fus $ matches the addition law 
in the Witt group, and is implemented by multiplications 
\[
\fus_w\boxtimes \fus_{w'} \to \fus_{ww'}. 
\]
To execute these operations, we must use isomorphisms $T(w)\boxtimes T(w') \cong T(ww')$ in~$\braid$. 
For instance, if $w'=w^{-1}$, and we have opted for the reverse-braided category $T(w)^{rev}$ to 
represent $w^{-1}$, the tensoring operation must be done \emph{over $T(w)^{rev}$}. 
This is how the object $X=T(w)$ satisfies $X\boxtimes X^\vee \equiv T(w)$, by means of an equivalence 
as \emph{algebra objects in $\fus$}.  This attempt meets obstructions and ambiguities 
from the incoherent choice of representatives $T(w)$. 

The obstruction problem is controlled by cocycles valued in the (higher) groups of automorphisms of 
the invertible $\fus$-module categories $\fus_w$. These all agree with the group $B^3\bC^\times$  
of units of $\fus$. More precisely, the Witt group~$W$ has a natural categorification to the $4$-group 
$\mathrm{GL}_1(\braid)$ of invertible BFCs, 
and we have 
\[
\pi_0\, \mathrm{GL}_1(\braid) = W, \qquad \pi_4\, \mathrm{GL}_1(\braid) = \bC^\times,
\] 
the other groups being zero because of the absence of non-trivial invertible objects in~$\vect, \cats$ 
and~$\fus$.\footnote{Uniqueness of $\vect^\otimes$ as invertible object in~$\mrE\fus$ follows 
from the multiplicativity of the Frobenius-Perron dimension.} The symmetric tensor structure 
on BFCs defines an $E_\infty$-map (up to isomorphism)
\begin{equation}\label{omega}
\varphi: \bS\otimes W \to \mathrm{GL}_1(\braid)
\end{equation}
from the `stable spherical Witt group' $\bS\otimes W$, an $\bS$-module which we define from a free resolution 
\begin{equation}\label{freeres}
F_1 W \to F_0W \to W.
\end{equation}
This allows the construction of a symmetric monoidal crossed-product category
\[
\left(\bS\otimes W\right)\ltimes^\varphi \fus,
\] 
involving the $3$-categorical truncation of $\bS$, which we wish to descend to $E'\fus$, 
expressed as an Eilenberg-MacLane crossed product $HW\ltimes \fus$.

This descent  requires an $E_\infty$ factorization of $\varphi$ through the Hurewicz map 
$H: \bS\otimes W \to HW$. The obstruction problem thus concerns the mystery arrow $m$ in the diagram
\begin{equation}
\label{omegalift}
\begin{gathered}
\xymatrix{
& HW\ar@{-->}[dr]^m & \\ 
\bS\otimes W \ar[ur]^H \ar[rr]^\varphi & & \mathrm{GL}_1(\braid).
}
\end{gathered}
\end{equation}
The obstruction is a stable $k$-invariant $k^5_s \in H^5_s(HW; \bC^\times)$, and $m$ admits a torsor 
of choices over the stable group $H^4_s(HW; \bC^\times)$. 
Using~$\bZ$, in first instance, instead of~$W$, to identify the obstruction problem, 
the relevant groups are the stable mapping spaces 
\begin{equation}
\label{bosonkinv}
[H\bZ; \Sigma^3 H\bC^\times] = 0, \quad
[H\bZ; \Sigma^4 H\bC^\times] = \bZ/6, \quad
[H\bZ; \Sigma^5 H\bC^\times] = 0,
\end{equation}
with the middle one generated by the Steenrod operation $Sq^4 \times P^1_3$ on $\bZ/2\times \bZ/3$. 

Use the free resolution \eqref{freeres} to identify the obstruction complex for $m$ with 
$Sq^4 \times P^1_3$ of the resolution for  
$\mathbf{R}\mathrm{Hom}\left(W;\bmu_6\right)$. Obstructions to $m$ are thus uniquely removed by passing to the natural 
extension $e: \widetilde{W}\to W$ by $\bmu_6$, for which 
\[
k^5_s = \left(Sq^4 \times P^1_3\right)\circ e,
\]
solving our obstruction problem of~\eqref{omega}, and completing the construction of~$E\fus$, via the diagram 
\begin{equation}
\label{omegalift2}
\begin{gathered}
\xymatrix{
& H\widetilde{W}\ar[dr]^ {\widetilde{m}} & \\ 
\bS\otimes \widetilde{W} \ar[ur]^H \ar[rr]^\varphi & & \mathrm{GL}_1(\braid).
}
\end{gathered}
\end{equation}
This construction gives the asserted direct sum decomposition 
$\mrE\fus = \bigoplus_{\widetilde{w}\in\widetilde{W}} \fus_{\widetilde{w}}$. 

Full dualizability follows from the realization of summands as fusion categories over  BFCs. 
The classification of invertibles in Property~(vi) follows from the uniqueness of 
$\vect^\otimes\in \fus$ as a Morita-invertible fusion category: an invertible in~$\fus_{\widetilde{w}}$ 
trivializes the class~$[T(w)]$ in the Witt group. Property (vii) stems from the relevant 
bordism group $\bZ/6$, a quotient of $\pi_3^s$: see Appendix~\ref{tangents}.  

To see the characterizing universality, choose a suitable target $3$-category~$\bT$, 
with the distinguished functor~$\Phi_1:\fus\to\bT$. 
Our $\mrE\fus$ will only see those objects in $\bT$ for which the $1/2/3$-portion of the 
generated TQFTs take values in $\Omega\fus=\cats$. In particular, their Drinfeld centers, 
a priori in~$\Omega\bT$, are finite sums of non-degenerate BFCs.  

Given a $\widetilde{w}\in\widetilde{W}$, choose a representative $X\in \mrE\fus$ with center~$T(w)$. 
The $1/2/3$-portion of~$\cT_X$ lands in~$\cats$, so~$X$ has a unique partner~$X'\in \bT$ generating 
that same TQFT in those dimensions. There is an isomorphism $\Phi_1(X\boxtimes X^\vee) \xrightarrow
{\ \sim\ } X'\boxtimes (X')^\vee$ in~$\bT$, inducing the identity on their common 
center~$\cT_X(S^1)= T(w)$, because the two objects generate the same Turaev-Viro $1/2/3$-theory. 
The latter is isomorphic, as an algebra object of~$\fus$, to $X\boxtimes X^\vee$ with its contraction 
algebra structure (see Remark~\ref{fakefusion}), and acts compatibly on~$X$ and~$X'$. Define an 
$\fus$-linear functor $\Phi_{\widetilde{w}}: \fus_{\widetilde{w}} \to \bT$ by
\begin{equation}\label{primefunctor}
Y\mapsto \Phi_{\widetilde{w}}(Y)=Y':= \left(Y\boxtimes X^\vee\right)\boxtimes_{\,T(w)} X'.
\end{equation}
We claim that $Y'$ generates the $1/2/3$-part of the theory $\cT_Y$. Indeed, we have in~$\mrE\fus$ 
that 
\[
Y\boxtimes X^\vee \equiv F, \quad Y\equiv F\boxtimes_{\,T(w)} X
\]
for a fusion category~$F$ with central~$T(w)$-action, so that~$\cT_Y$ is the result of 
sandwiching~$\cT_F$ and~$\cT_X$ with filling~$T(w)$. Since $Y'\equiv F\boxtimes_{\,T(w)}X'$, 
we obtain~$\cT_{Y'}$ by sandwiching~$\cT_F$ and~$\cT_{X'}$ with filling~$T(w)$. The claim 
follows from the agreement of the $1/2/3$-parts of~$\cT_X$ and~$\cT_{X'}$. 

The isomorphism class of~$\Phi_{\widetilde{w}}$ is independent of the representative~$X$, 
showing its uniqueness (as linear functor): 
when~$Y$ is simple, we may partner it with the~$Y'$ of~\eqref{primefunctor} (because of 
its uniqueness up to isomorphism in~$\bT$) and obtain by the same method a canonically 
isomorphic functor. The~$\Phi_{\widetilde{w}}$ assemble to a functor $\Phi: \mrE\fus \to \bT$.  
By construction, it is compatible with the symmetric tensor structure up to coherence 
isomorphisms. The coherences are pushed out from~$\mrE\fus$ by the universality of the 
lifting~\eqref{omegalift2}.
\end{proof}

\begin{remark}
We do not determine the extension class~$e$: it is what it is. For 
instance, it is known from~\cite{mug} 
that~$W$ has no $3$-torsion, so the $\bmu_3$-component is a split extension. 
\end{remark}

\begin{remark}[Symmetry] \label{projsym}
Our homotopy calculation can be made more insightful as follows. The fiber of the Hurewicz morphism 
$\bS \to H\bZ$ is the infinite loop space $\Omega^\infty \bS_{>0}$. By the Barratt-Priddy-Quillen theorem, 
this is the classifying space of the homotopy-abelianized 
symmetric group. The lifting problem~\eqref{omegalift} requires a \emph{trivialization} of $\varphi$ on 
$\Omega^\infty \bS_{>0}$, that is,  of the symmetric group action. This  is executed in  diagram~\eqref
{omegalift2}. Without it, we get a \emph{projective} symmetric monoidal structure on~$E'\fus$; 
see~\cite{projsym}.   

The $\bmu_6$-ambiguity in~\eqref{omegalift2} is explained similarly. The symmetric group action 
on powers of~$U$ defines a stable map $s_U:\Omega^\infty\bS_{>0} \to \Sigma^3H\bC^\times$, the classifying 
space of the unit fusion category. The Hurewicz fibration $\bS\to H\bZ$ and the vanishing of the 
relevant cohomology of $\bS$ give 
\[
[\Omega^\infty\bS_{>0}; \Sigma^3H\bC^\times] 
\cong [\Sigma^{-1}H\bZ; \Sigma^3H\bC^\times] = \bZ/6,
\]
under which the map $s_U$ becomes the generator.
\end{remark}

\begin{remark}[Universality] \label{univ}
Without the uniqueness requirement in~(viii), we could forgo the relation $U^{\otimes 6} \equiv  \vect^\otimes$ 
in $\mrE\fus$ and opt for an extension of $W$ by $\bZ$ instead of $\bZ/6$. The undesirable consequence 
is to create many instances of each $1/2/3$-truncated RT theory, differing only by powers of the 
phantom object~$U ^{\otimes 6}$ on the value of a point. Our construction instead ensures that distinct 
TQFTs are already distinguishable on their $1/2/3$ portion.  
\end{remark}

\begin{remark}[Decomposition into simples] A \emph{simple object} $X\in \mrE \fus$ is defined by the 
condition $\cT_X(S^2) =\bC$ (triviality of the second center). For $X\in \fus$, this agrees with 
indecomposability of the fusion category; moreover, every object in $\fus$ is a direct sum of 
simples. This generalizes to fusion categories over an invertible object $T\in \braid$; in fact, 
simplicity is detected by forgetting the $T$-action, because a splitting of the fusion category 
decomposes its Drinfeld center accordingly. This settles the direct sum decomposition into simples 
also for the new components of $\mrE \fus$. 
\end{remark}

\begin{remark}[`Fake fusion' calculus] \label{fakefusion}
A simple object $X\in\mrE\fus$ with center $T:=\mathrm{End}(\mathrm{Id}_X)$ lies in the 
$\mrE\fus$-component $\fus_T$, where it is represented by $T\otimes U^{\otimes k}$, where 
$T$ is viewed as fusion category over $T$, and the $U$-cofactor modifies the symmetry as explained 
in Remark~\ref{projsym}. We have a natural isomorphism of \emph{algebra objects in $\fus$} 
(with the $U$-cofactors cancelling out)
\[
X\boxtimes X^\vee \equiv T,
\]
with their natural actions on $X\in \fus_T$. This determines the calculus on~$X$: 
for instance, the action of a group~$G$ on~$X$ is equivalent to a homomorphism from~$G$ to the 
$3$-group of invertible $T$-modules. When the latter is a fusion category, that is the (higher) 
Brauer-Picard group of $X$  \cite{eno}. We will exploit this when discussing orientability (Theorem~\ref{thm5}). 

More generally, if $X_1$ and $X_2$ are objects with centers~$T_1$ and~$T_2$, 
then $\mathrm{Hom}_{\mrE\fus}(X_1, X_2) =0$, unless $T_1$ and $T_2$ are Witt-equivalent and the $X_i$ 
sit over the same point in the $\bmu_6$-extension torsor over $[T_1] =[T_2]$. In the latter case, 
an equivalence $T_1\boxtimes T_2^{rev} \equiv Z(F)$ identifies the $\mathrm{Hom}$ $2$-category 
with that of semi-simple $F$-module categories.
\end{remark}

\begin{remark}[$k$-invariant of $\mathrm{GL}_1(\mrE\fus)$]
\label{Smodeta}
The (higher) group of units of $\mrE\fus$ has the two non-zero homotopy groups 
$\pi_0 =\bZ/6$, $\pi_3=\bC^\times$, related by the universal $k$-invariant $Sq^4\times P^1_3$. 
More precisely, in terms of the co-fiber $\bS/\eta$ in the sequence
\[
\bS^1/2 \overset {\ \eta\ }{\rightarrowtail}\bS 	\twoheadrightarrow\bS/\eta,
\]
the spectrum $\mathrm{GL}_1(\mrE\fus)$ is the connective cover of the \emph{$3$-shifted Pontrjagin dual} 
$\mathrm{Map}\left(\bS/\eta; \Sigma^3\bI_{\bC^\times}\right)$.

The class $\eta$ represents the Koszul sign rule in the symmetric tensor structure on super-vector 
spaces, and killing it reflects the fact that we do not allow odd vector spaces in our target. 
This gives credence to $\mrE\fus$ as the universal target for \emph{bosonic} $3$-dimensional TQFTs.
\end{remark}

\begin{remark}[$4$ dimensions]
Adding fake fusion categories from $\mrE\fus$ as $1$-morphisms to the $4$-category $\braid$ kills 
the Witt group, and produces a $4$-category with spectrum of units $B \mathrm{GL}_1(\mrE\fus)$. 
This matches the ($4$-shifted) Pontrjagin 
dual of $\bS/\eta$: $\pi_4\left(\bS/\eta\right)=0$. However, this is not quite a universal target 
for $4$-dimensional bosonic TQFTs, as one can reasonably add  objects to $\braid$, such as the fusion 
$2$-categories of~\cite{f2cat}, and their generalizations using objects from $\mrE\fus$, as illustrated 
in Appendix~\ref{exgauge}. Our construction applies equally well to such enlargements of $\braid$, with their 
groups of invertible isomorphism classes replacing $W$. We plan to return to this in a follow-up paper.   
\end{remark}

%%%%%%%%%%%%%%%%%%%%%%%%%% SECTION %%%%%%%%%%%%%%

\section{Main result: fermionic case}
\label{fermion}
The main theorem has a natural generalization when super-vector spaces are included. Recall that the 
tensor structure on the category $S\vect^{\otimes}$ of super-vector spaces is symmetric, under 
the Koszul sign rule. This defines an (also symmetric) tensor structure on the $2$-category 
$S\cats$ of finite semi-simple module categories 
over $S\vect^{\otimes}$.

\begin{definition}
A finite, semi-simple \emph{super-category} is one equivalent to the category of modules 
(in super-vector spaces) over a finite complex semi-simple super-algebra. \emph{Functors} between super-categories 
are required to be linear over $S\vect^{\otimes}$. Tensor products are taken over $S\vect^{\otimes}$. 
A \emph{fusion super-category} is an $S\vect^\otimes$-algebra in $S\cats$ in which all objects 
have internal left and right duals. Braidings and $S\braid$ are defined as expected. 
\end{definition} 

\begin{remark}
We refer to \cite{ft3} for a wider discussion, but summarize the main points here:
\begin{enumerate}\itemsep0ex
\item Finite-dimensional, complex, simple algebras in super-vector spaces are Morita equivalent 
to one of $\bC$ or $\cliff(1)$. A simple object in a super-category in $S\cats$ therefore generates an additive 
summand of either $S\vect$ or of the regular $\cliff(1)$-module. 
\item Just as in the bosonic case, we use `fusion' where some authors use `multi-fusion'; every 
indecomposable fusion super-category is isomorphic (in $S\fus$) to one with simple unit. 
\item Indecomposable fusion super-categories have non-degenerate Drinfeld centers, and every fusion 
super-category splits into a direct sum of indecomposables. 
\item A fusion super-category $F$ cannot be pivotal if it includes $\cliff(1)$-lines. Moreover, its 
Drinfeld center $\mathrm{End}_{F\text{-}F}(F)$ and co-center $F\boxtimes_{F\boxtimes F^{op}}F$ 
can be inequivalent.  
\item As a result, Spin structures are needed in the respective TQFTs.
\end{enumerate}
\end{remark}

\begin{maintheorem}
There exists a symmetric monoidal $3$-category $\mrE S\fus \equiv \bigoplus_{\widetilde{w}\in\widetilde 
{SW}} S\fus_{\widetilde {w}}$ 
satisfying properties (i)-(vii) in~\S\ref{results}, with~$S\fus$ replacing~$\fus$ and $\bmu_{24}$ replacing 
$\bmu_6$. It is unique up to isomorphism, and universal as in Theorem~1.
\end{maintheorem}

\begin{proof}
The argument is the same, \emph{mutatis mutandis}. The principal change is that the group 
$\mathrm{GL}_1(S\fus)$ acquires two new homotopy groups $\pi_2= \pi_1 = \bZ/2$, in addition to 
$\pi_3 = \bC^\times$ at the top, from the odd lines and odd Clifford algebras, respectively. 
Specifically, $\mathrm{GL}_1(S\fus)$ is the connected cover of the shifted Pontrjagin dual to $\bS$:
\[
\mathrm{GL}_1(S\fus) = \big(\Sigma^3\bI_{\bC^\times}\big)_{>0}.
\]
There is a similar higher group $\mathrm{GL}_1(S\braid)$, having as $\pi_0$ the super-Witt group $SW$ of 
invertible isomorphism classes in $S\braid$. The obstruction problem is defined by the same 
diagram~\eqref{omegalift}, with $SW, S\fus, Sm$. However, the relevant homotopy groups are now
\begin{equation}
\label{fermionkinv}
\left[H\bZ; \big(\Sigma^3 \bI_{\bC^\times}\big)_{>0}\right] 
	= 0, 	\quad
\left[H\bZ; \big(\Sigma^4 \bI_{\bC^\times}\big)_{>1}\right] 
	= \bZ/24, \quad
\left[H\bZ; \big(\Sigma^5 \bI_{\bC^\times}\big)_{>2}\right] = 0.
\end{equation}
Indeed, we can determine them from the fibration sequence 
\[
\Sigma^{k-4}H\bZ/24 \rightarrowtail 
\left(\Sigma^k \bI_{\bC^\times}\right)_{>k-3} 
\twoheadrightarrow 
\left(\Sigma^k \bI_{\bC^\times}\right)_{>k-6 },
\]
which holds for all $k\in \bZ$, because of the vanishing $\pi_4^s = \pi_5^s =0$. Combining this 
for~$k=3,4,5$ with the vanishing of $[H\bZ; \Sigma^k\bI_{\bC^\times}]$ for $k>0$, we are led 
from~\eqref{fermionkinv} to the groups  
\[
\left[H\bZ;\Sigma^{-1} H\bZ/24\right] 
	= 0, 	\quad
\left[H\bZ; H\bZ/24 \right] 
	= \bZ/24, \quad
\left[H\bZ; \Sigma H\bZ/24\right] = 0.
\]

As before, a free resolution $F_1SW \to F_0SW \to SW$ converts the obstruction complex for our 
desired lifting of $S\varphi$ in~\eqref{omegalift} into the resolution of 
$\mathbf{R}\mathrm{Hom}\left(SW;\bZ/24\right)$. 
The obstruction problem is canonically resolved by a central extension 
\[
\bZ/24 \rightarrowtail \widetilde{SW} \twoheadrightarrow SW, 
\]
allowing us to define
\[
\mrE S\fus = 
\bigoplus\nolimits_{\widetilde{w}\in \widetilde{SW}} \,(S\fus)_{T(w)}.
\]
The universal properties are seen in the same way as for $\mrE\fus$. 
\end{proof}

\subsection{Invertible TQFTs.}\label{swcenter}
The new $\mrE S\fus$  contains $24$ isomorphism classes of invertibles, versus only the unit 
$S\vect^{\otimes}$ in ~$S\fus$. The invertibles represent the $24$ invertible framed 
TQFTs in dimension $3$, matching the $24$ distinct possible symmetric monoidal structures 
on an invertible object. The center $\bmu_{24}$ of $\widetilde{SW}$ is naturally identified with 
the Pontrjagin dual of~$\pi_3^s$, and the generating theory~$\ferm$ (see \S\ref{alphaomega}), 
which defines the standard character $\pi_3^s\to \bC^\times$, is characterized by a symmetry akin 
to that for $U$, in Remark~\ref{projsym}.

%%%%%%%%%%%%%%%%%%%%%% SECTION %%%%%%%%%%%%%%%%%%%%%%%

\section{Anomalous TQFTs and reduced central charge}
\label{spinomaly}
We discuss here the \emph{anomalous} versions for the TQFTs $\cT_X$ defined by objects 
$X\in \mrE S\fus$, their linearizations in tangential structures, and their central charges. 
We focus on $\Spin^{rp_1}_n$-structures for $n=2,3,\infty$, with tangent bundles classified 
by the homotopy fibers of $rp_1: B\Spin_n \to \Sigma^4H\bZ$. Integrality confines 
$r$ to $\frac{1}{4}\bZ$ for $n=2,3$, but only to $\frac{1}{2}\bZ$ for $n=\infty$. 
Oriented structures are discussed in the next section. 

On $3$-manifolds, $\Spin^{p_1/4}_3$-structures are equivalent to \emph{$3$-framings}, 
while $\Spin^{p_1/2}_3$-structures are equivalent to \emph{stable framings}. This is because 
$B\Spin^{p_1/4}_3$ is $4$-connected, while $B\Spin^{p_1/2}_3$ agrees with the homotopy fiber of 
$B\SO_3 \to B\SO$ through dimension $4$. Going further, $\Spin^{p_1}_3$-structures are a 
step towards the $\Spin^{\bC p_1}_3$-structures we discuss in~\S\ref{Qp1}. The reader may wish 
to consult Appendix~\ref{tangents} and the bordism groups in Table~\ref {mytable}. 

\subsection{Action of $\Spin_3$ via $\mu$.} \label{mu}
The Cobordism Hypothesis yields a change-of-framing action of $\mathrm{O}_3$ on the  $3$-dualizable 
objects and morphisms in the $3$-category $\mrE S\fus$. In our case, those assemble to the full 
underlying groupoid $\iota_{>0}\mrE S\fus$, because every object is $3$-dualizable \cite{bjs}. 
Upon restriction, the group~$\Spin_3\to \mathrm{O}_3$ must act via its lowest homotopy group 
$\pi_3 =\bZ$, defining a multiplicative $3$-automorphism of the identity of~$\mrE S\fus$. This gives a 
$3$-automorphism~$\mu(X)$ on each object~$X$, which is a scalar when~$X$ is simple. 

A unit local change in $3$-framing acts via the generator of $\pi_3\Spin_3$, and we conclude 

\begin{theorem}\label{thm3}
In the $3$-framed TQFT $\cT_X$ defined by a simple object~$X$, $3$-morphisms transform by~$\mu(X)$ 
under a unit local change in $3$-framing. Lower morphisms are unchanged, up to isomorphism. 
Moreover, $\mu(X)$ is multiplicative in~$X$.\qed 
\end{theorem}
\noindent
An object $X\in \mrE S\fus$ is $\Spin_3$-invariant precisely when $\mu(X) =1$. 
In the next section, we show this to be so when $X=F \in S\fus$; for fusion categories, this was 
conjectured in \cite{dsps}. Then,~$\cT_F$, \emph {a~priori} defined on $3$-framed 
manifolds, factors uniquely through $\Spin$-manifolds. 

\begin{proposition}\label{kermu}
$\mu$ surjects the center $\bmu_{24}\subset \widetilde{SW}$ 
onto the $12^{th}$ roots of unity $\bmu_{12}\subset \bC^\times$. 
\end{proposition}
\begin{proof} The invertible theories~$\psi^{\otimes k}$ are defined, as framed theories, from 
the~$24$ units in $\mrE S\fus$. The complete obstruction to their $\Spin_3$-invariance is then the differential 
\[
d_4: \left(\pi_3^s\right)^\vee \to H^4\big(B\Spin_3; \bC^\times\big) \cong \bC^\times
\]
in the Atiyah-Hirzebruch spectral sequence computing $\bI^*_{\bC^\times} MT \Spin_3$:  
there are no other differentials in this corner of the sequence. The computation of $\pi_0MT\Spin_3 =\bZ/2$ 
(Appendix~\ref{tangents}) shows 
that $d_4$ has kernel $\bZ/2$ and image $\bmu_{12}$.
\end{proof}

\subsection{Anomalous and linearized theories.}
The projective~$\Spin$ invariance of a simple~$X$ allows 
a description of~$\cT_X$ as an 
\emph{anomalous} TQFT for $\Spin$ manifolds, a structure we recall in~\S\ref{anomalybasics}. 
This anomalous presentation may be traded back for a $p_1$-related tangential structure plus 
a specified transformation law. Let $\alpha_{\mu(X)}$ (or simply $\alpha_X$) denote the 
$4$-dimensional invertible TQFT, defined on manifolds with $\Spin_3$ structure, valued in 
the spectrum $\Sigma^4\bI_{\bC^\times}$, and characterized by the closed manifold invariant 
$M\mapsto \mu(X)^{p_1(M)/4}$. The \emph{reduced central charge} $\underline{c}:= 
\frac{6}{2\pi i}\log \mu \pmod{6\bZ}$ could make $\alpha_X$ more familiar: 
\[
\mu(X)^{p_1/4} = \exp\left(2\pi i \underline{c}(X) \cdot \frac{p_1}{24}\right).
\] 
The following Theorem and Remark summarize the anomaly/tangential trade; its proof is contained in 
the discussion in~\S\ref{anomalybasics}--\ref{rels} below. 

\begin{maintheorem}\label{thm4}
The object $X\in\mrE S\fus$ defines an \emph{anomalous $3$-dimensional $\Spin$ theory} $\alpha\!\cT_X$, 
a boundary theory for the anomaly theory $\alpha_X$. 
We can linearize $\alpha\!\cT_X$ as follows:
\begin{enumerate}\itemsep0ex
\item $\alpha\!\cT_X$ is linearizable in two ways over $\Spin^{p_1/4}_3$-manifolds (framed) such that one 
step in $p_1/4$-structure changes $3$-morphisms by a factor of $\mu(X)$. 
\item After a choice $\mu(X)^{1/2}$, the theory $\alpha\!\cT_X$  is linearizable in two ways over 
$\Spin^{p_1/2}_3$-manifolds (stably framed) such that one step in $p_1/2$ structure changes $3$-morphisms 
by a factor of $\mu(X)^{1/2}$. 
\item After a choice of $\mu(X)^{1/4}$, the theory $\alpha\!\cT_X$  is linearizable in two ways over  
$\Spin^{p_1}_3$-manifolds such that one step in $p_1$ structure changes $3$-morphisms by a factor of 
$\mu(X)^{1/4}$.
\end{enumerate} 
In (iii), the anomaly theory $\alpha_X$ factors uniquely through oriented $4$-manifolds.
\end{maintheorem}

\begin{remark}
The trade carries cost (if~$X$ is forgotten), because of the automorphisms of~$\alpha_X$. 
These can be described in terms of the invertible theories~$\spit, \omega$, and~$\psi$ described 
in~\S\ref{alphaomega} below. For instance, attempting to reconstruct~$\cT_X$ from~$\alpha\cT_X$ 
loses the distinction between~$X$ and~$XU^{\otimes 3}$, even though the two define distinct framed theories. 
More generally, 
\begin{enumerate}\itemsep0ex
\item The two choices for $\cT_X$ in (i-iii) differ by a factor of~$\spit$.
\item In (ii), tensoring with $\omega$ flips the choice of square root of $\mu(X)$.
\item In (iii), we can cycle through choices of fourth root by powers of $\ferm^{\otimes 12}$. 
\end{enumerate}
The structures (ii-iii) involve lifting~$\underline{c}$ modulo~$12$ and~$24$, respectively.
\end{remark}
  
\subsection{Refresher on anomalous theories.}
\label{anomalybasics}
Anomalous TQFTs may be described as boundary theories for an invertible theory in one higher 
dimension; see for instance~\cite{ft2, fta}.
Invertible TQFTs map into the spectrum of units of the target category, and thus factor through 
\emph{stable} maps from the monoidal group completions of the (tangentially appropriate) bordism 
categories, the \emph{Madsen-Tillmann spectra}. Standard convention identifies the group-completions with the 
suspensions of~$MT$ spectra which start in degree~$0$, which accounts for some awkward degree shifts 
in our paper. When the units in the target category form 
(the connective cover of) the spectrum~$\Sigma^4 \bI_{\bC^\times}$, stable maps with that target are 
classified by the Pontrjagin dual of $\pi_4$ of the source. The invertible TQFT is then determined 
by the numerical invariants of the TQFT on closed manifolds. We refer to~\cite{fh} for the comprehensive 
account of these ideas. Our anomaly theory $\alpha_X$ is defined, in first instance, on $4$-manifolds 
with $\Spin_3$ structure, where $p_1/4$ is an integral class.

The natural home (target $4$-category) to use  for $\alpha_X$ is the \emph{delooping}~$B\mrE S\fus$, 
the symmetric monoidal category with a single, unit object\footnote{To give this a semblance of 
respectability, we can also include direct sums of copies of $\mrE S\fus$.}~$\mrE S\fus$. 
Its group of invertibles is precisely the 
connective cover of~$\Sigma^4 \bI_{\bC^\times}$. The point generator for~$\alpha_X$ is the unit 
object, making it trivial as a \emph{framed} theory. The anomalous theory is~$X$ itself, as 
a morphism from $\mrE S\fus$ to itself, but with the source and target carrying  different 
$\Spin_3$-invariance data: the standard one, and a $p_1$-twist. 

\begin{remark}
This delooping target works for any anomaly theory, but is somewhat unsatisfactory: 
we wish to land in (an enhancement of) $S\braid$, a universal target for $4$-dimensional TQFTs defined 
from suitable algebra objects in $\mrE S\fus$. There, the point generator of $\alpha_X$ is represented by 
the algebra object $\mathrm{End}(X)$ in $\mrE S\fus$; its anomaly for $\Spin_3$-invariance is naturally 
cancelled. The object $X$ then defines a boundary theory $\alpha\!\cT_X$ for $\alpha_X$ --- the standard module 
for its own endomorphism algebra.
\end{remark}

\subsection{The invertibles $\spit,\omega$ and $\ferm$.} 
\label{alphaomega}
Promoting $3$-framed theories to $p_1$-tangential structures meets some invertible ambiguities: 
\begin{enumerate}\itemsep0ex
\item The \emph{$3$-framed} theory $\spit$ is determined by its value $\spit(pt)= U^{\otimes 3}\in \mrE \fus$, 
the unit of order $2$. However, it factors uniquely through $\Spin_3$ manifolds, since the 
sign character of~$\pi_3^s$ factors through the natural map
\[
\pi_3^s\to \pi_0 MT\Spin_3\cong \bZ/2. 
\]
Alternatively, $\mu(U^{\otimes 3}) =1$, enforcing $\Spin_3$-invariance. See~\S\ref{alpha} below for a 
geometric description of the associated manifold invariant.  

\item \emph{Stably framed} $3$-manifolds carry an invertible order-two theory $\omega$, whose manifold 
invariant integrates the difference between two trivializations of $w_4$: from $3$-dimensionality, and from the 
stable framing.
 
Unlike $\spit$, the theory $\omega$ is trivial on \emph{$3$-framed} manifolds, where the two cancellations 
of $w_4$ agree. In particular, $\omega(pt) = \vect^\otimes\in \fus$. On the other hand, $\omega$ detects a unit shift 
in $p_1/2$ structure,  whereas the lift of $\spit$ to $\Spin^{p_1/2}_3$-manifolds is insensitive to that shift.

\item  We define the invertible \emph{free fermion theory}~$\ferm$ on $\Spin^{p_1}_3$-manifolds 
using the  invariant
\begin{equation}\label{psidef}
\pi_0MT\Spin^{p_1}_3 \cong \pi_3M\Spin^{p_1}\oplus \pi_0MT\Spin_3  \to \pi_3M\Spin^{p_1} \cong \bZ/48 
\rightarrowtail \bC^\times,
\end{equation}
coming from the first projection followed by the standard character. We will further extend~$\ferm$ to~$\ferm_\bC$ 
on $\bC p_1$-structures in~\S\ref{Qp1}, using the analogous splitting of $\pi_0MT\Spin^{\bC p_1}_3$. 

On framed manifolds, $\psi$ restricts to define the standard character $\pi_3^s\to \bC^\times$. Thus, $\ferm(pt)\in \mrE S\fus $ is the invertible object sitting over the 
generator of $\ker(\widetilde{SW} \to SW)$, and $U = \psi(pt)^ {\otimes 4}$. 
\end{enumerate}
\subsection{Relations.}\label{rels}
The theories $\spit, \omega,\psi$ and $\cT_U$ are related as follows:
\begin{itemize}\itemsep0ex
\item $\omega =\spit\otimes \ferm^{\otimes 12}$ on stably framed manifolds. \\
This formula also extends $\omega$ to $\Spin^{p_1}_3$-manifolds, and to $\bC p_1$-structures via $\ferm_\bC$, 
but then it no longer has order $2$. 
\item $\cT_U = \spit\otimes \ferm^{-8}$ on $\Spin^{p_1}_3$-manifolds, where we define $\cT_U$ 
on $\SO^{p_1}_3$-manifolds by the standard character $\pi_0MT\SO^{p_1}_3=\bZ/6\to \bC^\times$, 
and then lift to $\Spin_3^{p_1}$.   \\
As before, this formula extends $\cT_U$ to $\bC p_1$-structures, but it no longer has order $6$. 
\end{itemize}
Upon extension to $\Spin^{p_1}_3$-structures, the $\!\!\mod{6}$ reduced central charges $\underline{c}$ 
refine $\!\!\mod{24}$, reflecting the behavior under a one-unit shift of $p_1$-structure. 
Further extension to $\bC p_1$-structures (cf.~\S\ref{Qp1}) lifts the central charges to $\bC$. 
$\Spin$ invariance of $\spit$ and the relations above tell us that 
\[
c(\spit) = 0, \quad c(\ferm)= \frac{1}{2},\quad 
c(\omega) =6,\quad {c}(\cT_U) =-4\pmod{24}.
\] 

\subsection{Caution.} \label{caution}
We stress that the central charge~$c$ is defined in terms of coupling to (multiples of)~$p_1$, 
not to (re)framings. If this is not tracked correctly, the kernel $\{\pm1\}$ of~$\mu$ (Proposition~\ref
{kermu}) creates conflicts. For instance, the relation $U =\psi(pt)^{\otimes 4}\in \mrE S\fus$ may 
suggest that $c(\cT_U) =2$. Of course, $2=-4 \pmod{6}$, matching the answers on $3$-framed theories, 
but trouble comes from assuming that unit shifts in $\pi_3^s$ and in $p_1/2$ have the same effect on 
stable framings. The theory~$\spit$ breaks that link, and the maps in Proposition~\ref {spinmaps} 
show that $\cT_U \neq \psi^{\otimes 4}$ as $\Spin^{p_1}_3$-theories. Section~\ref{ffermion} shows 
other apparent inconsistencies around~$\ferm$, if its domain is not tracked carefully.

\subsection{The $\Spin_3$-manifold invariant.} \label{alpha}
As an invertible theory, $\spit$ is uniquely determined by its $3$-manifold invariant, detected on 
$\pi_0MT\Spin_3\cong\bZ/2$ as follows. For a closed, $\Spin$ $3$-manifold~$N$, choose a trivialization 
of the cocycle~$p_1/4$. Also choose a $\Spin$ $4$-manifold $M$ with $\partial{M} = N$, exploiting the 
vanishing of $\pi_3M\Spin$. The cocycle~$p_1/2$ on $M$ has been trivialized on~$\partial M$, so 
$\int_M p_1/2$ is an integer. A change in the boundary trivialization of~$p_1/4$ shifts the integral 
by an even number, so  $\spit(N):= \int_M p_1/2\mod{2}$ is well-defined. It vanishes if we can find 
a (stable) $3$-dimensional reduction of~$TM$, because $p_1/4$ is then an integral class. On the 
other hand, $\spit\neq 0$ on the $3$-sphere with Lie group framing.

\subsection{Loss of information.} As already mentioned, the anomalous version $\alpha\!\cT_X$ 
is less precise than $\cT_X$, since $\alpha_X$ can be trivialized more often than first meets the eye. Thus,  

\begin{proposition}
When $\underline{c}(X) \in\frac{1}{2}\bZ$, $\alpha\!\cT_X$ can be linearized to a $\Spin$ TQFT in two ways, 
differing by~$\spit$. 
\end{proposition}

\begin{proof}
When $\mu(X)\in \bmu_{12}$, Rohlin's theorem on the divisibility of the signature ensures the vanishing
of~$\mu(X)^{p_1(M)/4}$ on closed $\Spin$ $4$-manifolds, so that $\alpha_{\mu(X)}$ is trivializable; the 
trivializations form a torsor over $\langle\spit\rangle$. 
\end{proof}
\noindent
When $\underline{c}$ is \emph{half}-integral, there is no preferred choice; 
but a preferred one for integral $\underline{c}$ is obtained by squaring either choice of trivialization 
for $\underline{c}/2$.

%%%%%%%%%%%%%%%%%%%%%%%%%% SECTION %%%%%%%%%%%%%%%%%%%%%%%%

\section{Spin invariance of fusion super-categories.}
\label{spinv}
We now show that the invariant $\mu$ factors through the (centrally extended) super-Witt group. This settles a 
conjecture of \cite{dsps}. 

\begin{maintheorem} \label{mudef}
The action of $\Spin_3$ on $\mrE S\fus$ factors through a group homomorphism $\mu: \widetilde{SW} \to \bC^\times$.
In particular, a fusion super-category $F\in S\fus$ has $\mu\equiv 1$, and 
carries unique $\Spin_3$-invariance data.
\end{maintheorem}
  
\begin{remark}
The condition $[X]=[Y]\in \widetilde{SW}$ is equivalent to $X\boxtimes Y^\vee \in S\fus$, 
so the special case is  equivalent to the general statement. 
\end{remark}

\begin{proof}
This would be straightforward, if the $\Spin_3$-action extended to all $1$-morphisms in $\mrE S\fus$: 
indeed, $\pi_3$ would act trivially on $1$-morphisms, due to their categorical cutoff, and a 
fusion super-category $F$ is related to the unit $S\vect^\otimes$ by the regular module ${}_FF$, forcing 
the equality of projective obstructions:  $\mu(F)= \mu(S\vect^\otimes) =1$. 

While $\Spin_3$ need not \emph{a priori} act on the collection of \emph{all} $1$-morphisms 
in $S\fus$, we will use the (shifted) Hopf fibration $\Omega S^2 \rightarrowtail \Spin_2 
\twoheadrightarrow \Spin_3$ to read off~$\mu$ from $\Spin_2$, which \emph{does} act on $\iota_{>1}\mrE S\fus$. 
Every fusion super-category~$F$ is naturally invariant under~$\Omega S^2$, because the 
action of that group on~$\iota_{>0}\mrE S\fus$ factors through the trivial map to $\Spin_3$. 
Thus,~$\Omega S^2$ acts on $\mathrm{Hom}_{S\fus}(S\vect^\otimes; F)$. We claim that the regular 
module~${}_FF$ therein carries a natural $\Omega S^2$-invariance structure. This extends the action 
of~$\Spin_2$ to~$\Spin_3$, and completes the argument. 

For our claim, it suffices to trivialize the action of~$\pi_1\Omega S^2$ on~${}_FF$, 
compatibly with its natural trivialization on~$F$: there are no further obstructions to a section 
through~${}_FF$ over $S^2 = B\Omega S^2$. Now, $\pi_1\Omega S^2$ acts on $F$ and ${}_FF$ by their squared 
\emph{Serre automorphisms}. Invariance is the content of the following addition to Theorem~\ref{mudef}; 
this extends the action of~$\Spin_3$ to $1$-morphisms in $\mrE S\fus$ (albeit not in a way compatible with composition). 
\end{proof}

\begin{theorem}\label{spinboundary}
For $F\in S\fus$, every $F$-module $M\in \cats$ admits a trivialization of the square of 
its Serre automorphism~$S_M$ relative to $F$, compatible with the $\Spin_3$-enforced trivialization of~$S^2_F$ on~$F$. 
\end{theorem}

\begin{remark}
Underlying the theorem is a (potentially stronger) property internal to~$F$, which indeed is what we prove. 
The square of the Serre functor on~$F\in S\fus$ is the quadruple dual, identifiable with~$\mathrm{Id}_F$ 
\emph{as a tensor functor} using Radford's isomorphism~\cite {dsps}. Any other identification 
$S_F^2 \cong \mathrm{Id}_F$ as automorphisms of $F\in S\fus$ differs from Radford's by braiding 
with a central element~$z\in Z(F)$. The content of the theorem is that~$z$ maps to $\mathbf{1}\in F$ 
for the identification enforced 
by the projective $\Spin_3$-invariance of~$F$ (via $\pi_1\SO_3 = \bZ/2$). The same identification, 
acting on the regular module~${}_FF$, then gives our compatible trivialization of $S_F^2$ there. 
The case of other modules reduces to the regular one by Morita equivalences.
\end{remark}

\begin{proof}
View the Serre functor~$S_F$ of~$F$ as a co-oriented self-interface for~$F$ in a chosen ambient 
framing, performing a full framing twist upon crossing the interface. This is not a genuine 
framing defect, as it can be spread out into the bulk by deformation; but treating it as 
such, relative to the chosen framing, facilitates the argument. The framing jump across this 
interface can be ended in a genuine framing singularity, described by a tangent vector along 
the supporting line plus the radial framing in the normal directions. Our TQFT functor~$\cT_F$ 
is \emph{undefined} on this framing defect. (We will only need to evaluate it when 
investigating~$\SO_3$-invariance later in the paper.)

The squared defect~$S_F^2$ is also endable in a framing defect, now with a double-twist (dipole) 
singularity in the normal directions. This double twist is implemented by a big circle in 
$\Spin_3$. This is now trivializable by a contracting homotopy, which we choose once and for all. 
(The normal bundle to the line is framed by the Serre interface. Without the latter, the defect 
trivialization would be ambiguous on a closed loop, because the space of contracting homotopies is a 
torsor over~$\Omega^2 S^3$.) This defines an ending 
defect~$\partial S^2_F$ for~$S_F^2$ in~$\cT_F$, which on any linking circle is isomorphic to the 
unit object $\mathbf{1}\in Z(F)$, the \emph{transparent defect}. 

On the regular boundary theory (which we keep denoting~$M$ for notational clarity), the bulk Serre 
interfaces~$S_F, S^2_F$ can be terminated in interfaces~$S_M$ and~$S^2_M$, implementing a (now tangential) 
surface Serre framing twist and its square. Trivializing~$S^2_M$ compatibly with $\partial S_F^2$ 
means ending  it in an invertible defect, which will be a boundary endpoint of $\partial S_F^2$, as 
in Figure~\ref{bdrys2}: 

\begin{figure}[h]
\centering
\includegraphics[width=5in]{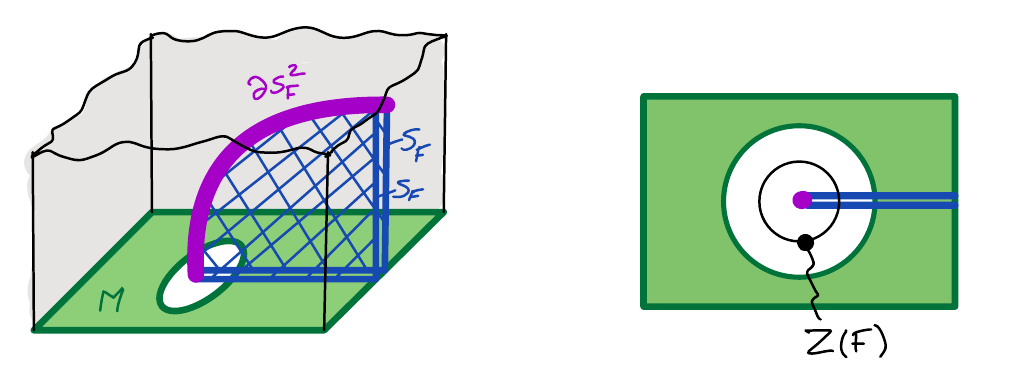} 
\caption{
The end $\partial S^2_F$ of the bulk squared-Serre reaching the boundary  
and its value in $Z(F)$}
\label{bdrys2}
\end{figure}
\noindent
We cannot end the defect \emph{geometrically}: the framing must \emph{a priori} be tangent to the 
boundary, which precludes a trivialization of the double twist. 

We therefore switch to algebraic calculus. The end~$\partial S^2_F$ of~$S^2_F$ 
produces the object~$\mathbf{1}\in Z(F)$ in a linking circle, so our desired ending defect lives 
in the $\cT_F$-space produced by the clean disk with $M$-boundary and doubly twisted boundary 
$2$-framing. We evaluate this by using the standard elbow of a solid cylinder with $M$-boundary, 
which identifies this with the dual output for the same picture, but now with the blackboard framing. 
If the fusion category~$F$ is simple with simple unit --- which we may arrange, for purposes 
of this theorem, by Morita equivalences and direct sum decompositions --- then this picture 
computes the vector space 
\[
\bC \cong \mathrm{Hom}_{Z(F)}(\mathbf{1}, W) = \mathrm{Hom}_{F}(\mathbf{1}, \mathbf{1})
\]
with the `Wilson object'~$W= \iota^*(\mathbf{1})$ in~$Z(F)$ (where $\iota:Z(F)\to F$ is the 
natural morphism). Any non-zero vector therein must define an isomorphism between $\mathrm{Id}_M$ 
and~$S_M^2$, as both  are invertible, and therefore simple objects in $\mathrm{End}_F(M)$.  
This gives the sought-after trivialization of $S_M^2$. 
\end{proof}

%%%%%%%%%%%%%%%%%%%%%% SECTION %%%%%%%%%%%%%%%%%%%%%%%%

\section{Orientability and modular structures}
\label{orient}

When~$X\in \mrE\fus$ is supplemented by a modular tensor structure on $\cT_X(S^1_b)$ (for the circle with  
bounding $3$-framing), we can forgo the $\Spin$ structure and factor $\cT_X$ through 
\emph{oriented}~$p_1$-manifolds. Now, the anomaly theory $\alpha_X$ of~\S\ref{spinomaly} has~$6$ 
trivializations as an $\SO^{p_1}_3$-theory, related by the powers of~$U$. This ambiguity in 
reconstructing $\cT_X$ from Walker's anomalous model is removed by remembering the point generator~$X$.  

\begin{maintheorem}
\label{thm5}
For $X\in\mrE\fus$, a modular structure on the braided fusion category $T:=\cT_X(S^1_b) $ descends $\cT_X$ uniquely to oriented 
manifolds with $p_1$-structure. Thereby, $\mu(X)$ acquires a preferred fourth root. 
\end{maintheorem}

\begin{remark}
For $1/2/3$-manifolds, Turaev's construction of TQFTs \emph{with signature structure} is 
a version of this result; an account is also found in \cite{bdsv}. Interpretation requires some care, though: 
Turaev's construction can only be refined to a symmetric monoidal functor between \emph{projectively symmetric} 
monoidal categories (see our upcoming discussion in~\cite{projsym}). A sign ambiguity in that construction 
comes from tensoring with $U$, which has central charge~$(-4)$. 
Knowledge of~$X$ and the use of $p_1$-structures 
refines the central charge mod~$24$.
\end{remark}

\subsection{Preliminaries.} 
When a connected group $G$ acts on a topological space $S$, a choice of base-point $x\in S$ leads to a group extension
\begin{equation}
\label{groupext}
\Omega_x(S) \rightarrowtail \Omega_x(S_G) \twoheadrightarrow G,
\end{equation}
as part of the fibration sequence that continues with 
\begin{equation}
\label{classext}
G\xrightarrow{\ \cdot x\ } S \to S_G \to BG.
\end{equation}
\emph{Invariance data} for~$x$ is a homotopy $G$-fixed point structure: a section of the map 
$S_G\to BG$ through the point $x$. 
This is equivalent to a section of the fibration~\eqref{groupext} \emph{as a group homomorphism}. 

We will produce invariance data for $X\in \mrE\fus$ under the group $\SO^{p_1}_3$ for precisely 
one choice of $\mu(X)^{1/4}$. Recall that $\SO^{p_1}_3$ is the fiber of $\Omega p_1: \SO_3\to K(\bZ;3)$. 
Given the absence in $\mrE\fus$ of $(k>3)$-morphisms, the action factors through the Postnikov 
truncation to the group~$G$ with only
\[
\pi_1G =\bZ/2,\qquad \pi_2 G = \bZ/4,
\]
extended in $BG$ by the \emph{Pontrjagin square} $\wp: K(\bZ/2;2) \to K(\bZ/4;4)$. 

\begin{proof}[Proof of Theorem~\ref{thm5}]
On $S=\iota_{>0}\mrE\fus$, the group~$G$ acts via $\mathrm{O}_3$; with $x = X$ in the above 
discussion and writing~$GA_X$ for $\Omega_x(S_G)$, we get an extension
\begin{equation}
\label{groupext2}
\mathrm{Aut}_{\,\mrE\fus}(X) \rightarrowtail 
	GA_{X} \twoheadrightarrow G,
\end{equation}
and a $G$-fixed point structure on $X$ is a group splitting $G\to GA_X$ of the last map. 

The group $\mathrm{Aut}_{\,\mrE\fus}(X)$ is determined by 
$T=\cT_X(S^1_b)$ (Remark~\ref{fakefusion}) and has 
homotopy groups 
\begin{equation}
\label{htpyX}
\pi_0\, \mathrm{Aut}_{\,\mrE\fus}(X) = [\mathrm{Aut}^{br} T], 
	\quad \pi_1 = [\mathrm{GL}_1 T], 
	\quad \pi_2 = \bC^\times,
\end{equation}
(\emph{isomorphism classes} of braided tensor automorphisms, invertibles, and scalars). The group~$GA_X$ 
is the homotopy quotient
\[
\Omega G \rightarrowtail \mathrm{Aut}_{\,\mrE\fus}(X) \twoheadrightarrow GA_X. 
\]
By construction of $G$, the action of $K(\bZ/4;2)\subset G$ on $\iota_{>0}\mrE\fus$ factors as the (stable) map
\begin{equation} \label{classpi2}
K(\bZ/4;2)\xrightarrow{\ B_4\ } K(\bZ;3) \xrightarrow{\ \mu\ } K(\bC^\times;3),
\end{equation}
ensuring the vanishing of the map 
\[
\bZ/4 = \pi_1\Omega G \to \pi_1 \mathrm{Aut}_{\,\mrE\fus}(X) = [\mathrm{GL}_1 T].
\]
A splitting of~\eqref{groupext2} requires the vanishing of the map $\pi_0\Omega G \to [\mathrm{Aut}^{br} T]$. 
If that is the case, then the homotopy groups of $GA_X$ are 
\[
\pi_0 GA_X = [\mathrm{Aut}^{br} T],
\]
while $\pi_1$ and $\pi_2$ are the group extensions induced by the map $GA_X\to G$ in~\eqref{groupext2}:
\begin{eqnarray}\label{groupext3}
	\left[\mathrm{GL}_1 T\right] \rightarrowtail 
	&\pi_1 GA_X& \twoheadrightarrow \pi_1G 
	= \bZ/2,\\
	\bC^\times \rightarrowtail &\pi_2 GA_X& 
	\twoheadrightarrow \pi_2G = \bZ/4. \label{groupext4}
\end{eqnarray}
The second one, classified by \eqref{classpi2}, has four splittings, matching the four 
choices of~$\mu(X)^{1/4}$. 

A group splitting of \eqref{groupext2} is then equivalent to  
\begin{enumerate}\itemsep0ex
\item the vanishing of the connecting map $\pi_1G \to [\mathrm{Aut}^{br} T] $ in~\eqref{groupext2},
\item a splitting of the consequent group extension~\eqref{groupext3} of $\pi_1G$,
\item a splitting of the resulting extension of $G$ by $B^2\bC^\times$.
\end{enumerate}

Now, once (i) and (ii) have been addressed, (iii) has a unique solution, because
\begin{equation}\label{unique}
H^3(BG;\bC^\times) = H^4(BG;\bC^\times) =0.
\end{equation} 
More specifically: a map $K(\bZ/2;2) \to K(\bC^\times;4)$, classifying a group extension of the base 
$B\bZ/2$ of~$G$ by  $B^2\bC^\times$, factors uniquely through the Pontrjagin square~$\wp$ to~$K(\bZ/4;4)$, the 
$k$-invariant of~$BG$. Every extension of $G$ by $B^2\bC^\times$ is then split by exactly one 
of the four splittings of~\eqref{groupext4}. 
Each choice also decouples the actions of the two Pontrjagin layers of~$G$, after we push 
out~$\pi_2G$ into $B^3\bC^\times$ through $\mu(X)\circ B_4$. We must then only handle 
items~(i) and~(ii) above, and they only concern  $\pi_1G=\pi_1\SO_3=\bZ/2$. 

\subsection*{Vanishing of the connecting map.} We refer to \cite{bk, dsps, hpt, pen} for results 
on braided fusion categories. On $X$, $\pi_1\SO_3$ acts by the Serre automorphism~$\Serre_X$, equipped 
with a trivialization of its square. In $\mathrm{Aut}^{br}T$, this becomes the square of the 
braiding~$\beta$: the identity functor on~$T$, with a braided automorphism of the multiplication.\footnote {Drinfeld's isomorphism of objects with their 
double duals identifies $S$ with the internal double dual functor, the Serre functor on $\cT_X(S^1)$ 
qua fusion category.} 
This functor, and the connecting map in~(i) along with it, is trivialized by a \emph{balancing twist}: an 
automorphism $\theta$ of the identity of the underlying category of $T$, related to the braiding $\beta$ 
by the identity
\[
\theta_{a\otimes b}\circ \left(\theta_a^{-1}
	\otimes \theta_b^{-1}\right) = \beta_{b,a}\beta_{a,b}, 
\quad\forall a,b\in T.
\]  
\subsection*{Splitting the extension~\eqref{groupext3}.}
From $\theta$, we extract a \emph{tensor} automorphism of $\mathrm{Id}_T$:
\[
\rho(\theta): a\mapsto \theta_a\circ (\theta_{a^*}^*)^{-1} \in \mathrm{End}(a), \quad\forall a\in T.
\] 
Non-degeneracy of $\beta$ ensures that $\rho$ is effected by the double-braiding with some 
$t\in \mathrm{GL}_1 T $:
\begin{equation}\label{dblbraid}
\mathrm{db}(t): =(t^{-1}\otimes)\circ\left(\beta_{x,t}\beta_{t,x}\right)\circ(t\otimes):x\xrightarrow{\ \sim\ } x.
\end{equation}
The isomorphism class of $t\in \mathrm{GL}_1T$ modulo squares  represents the extension class in~\eqref 
{groupext3}. If $t$ has a square root $r$, we can kill $\rho$ by composing $\theta$ with 
$\mathrm{db}(r)^{-1}$. 

The \emph{ribbon condition} on $\theta$, the final modularity constraint for $T=\cT_X(S^1_b)$, 
is precisely $\rho\equiv 1$, so such a $\theta$ provides the splitting required in (ii).    
\end{proof}

\begin{remark}
The complement $\pi_1GA_X\setminus \left[\mathrm{GL}_1 T\right]$ can be identified with the set of 
balancings: a torsor over $\left[\mathrm{GL}_1 T\right]$ under composition with $W$. 
Its addition law into $\left[\mathrm{GL}_1 T\right]$ is 
\[
\theta + \eta = t,\quad\text{defined by}\quad \mathrm{db}(t)(a) =  \theta_a \circ 
\left[(\eta_{a^*})^*\right]^{-1},\forall a\in T.
\]
We will meet this structure in the next section. 
\end{remark}

%%%%%%%%%%%%%%%%%%%%%%%%% SECTION %%%%%%%%%%%%%%%%%%%%%%%

\section{Spherical structures and canonical orientations}
\label{canorient}

\subsection{Synopsis.}
Call $X\in \mrE\fus$ \emph{$p_1$-orientable} if its center admits a modular structure, which we also call 
a $p_1$-orientation on $X$. For a $p_1$-orientable~$X$, we now describe a \emph {canonical} $p_1$-orientation 
(Definition~\ref{defcan}) with optimal properties. It is the unique splitting of~\eqref{groupext3} in 
which the lifted~$\pi_1G$ is central in~$T$. Denoting the resulting $\SO^{p_1}_3$-theory by~$\cT_X^1$,  
we will show: 
\begin{enumerate}\itemsep0ex
\item[(a)] Line operators in $\cT^1_X$ do not require $\Spin$ structures along their support;\footnote{They 
still require \emph{normal} framing information.}
\item[(b)]Interfaces between $\cT^1_X, \cT^1_Y$ can be made $\SO_2$-invariant: in particular, a non-zero 
morphism between simple objects forces their central charges to agree $\bmod{\:24}$; 
\item[(c)] When $X=F$ is a fusion category and $Y=\vect^{\otimes}$, the central charge of $\cT_F^1$ 
vanishes~$\bmod{\:24}$, and $\cT^1_F$, along with all its boundary theories, are definable on oriented manifolds.
\end{enumerate}
All other $p_1$-orientations arise from the canonical one by shearing a canonical splitting of~\eqref{groupext3} by 
elements $z\in [\mathrm{GL}_1T]$ of order~$2$. Thus oriented, we denote the theory by~$\cT_X^z$.  
This need not satisfy~(a-c) above: in fact, part~(a) characterizes the canonical $p_1$-orientation. 
When~$X$ is a $p_1$-orientable fusion category~$F$, the element~$z$ determines the reduced central charge 
of~$\cT_F^z$, namely $0\text{ or }12\mod{24}$, according to the value $\beta_{z,z}\in\{\pm1\}$ of 
its self-braiding. When $\beta_{zzz}=1$, $z$ controls which boundary theories can carry $\SO_2$-invariance 
data; when $\beta_{z,z}=-1$, none of them do.

Orientation-change on a fusion category~$F$ can be executed more geometrically by coupling~$\cT_F$ to 
the four versions of $\bZ/2$-gauge theory, as we will describe in \S\ref{changeorient}. 

\begin{remark}
When $X=F \in \fus$ is $\SO_3$-invariant, and~$M$ is a simple module category, $\SO_2$-invariance 
data for~$M$, if they exist, form a $\bC^\times$-torsor. Having chosen an algebra object $a \in F$ whose 
$F$-internal category of modules is equivalent to $M$, an invariance datum is equivalent to a Frobenius structure 
on $a$, and is determined by a single number, the trace of the algebra unit.     
\end{remark}  

\subsection{Main results.} Here is the classification of orientation structures on our TQFTs 
and compatible structures on boundaries and interfaces; some of the details must await clarification 
in the proof. Recall our notation $\vect^{br,\zeta}_{\bZ/2}$ for 
the four braided versions of $\bZ/2$-graded vector spaces. 

\begin{maintheorem}[Canonical $p_1$-orientation] \label{canonical}
{\ }
\begin{enumerate}\itemsep0ex
\item The $\SO^{p_1}_3$-action on the (full symmetric monoidal sub-)groupoid of $p_1$-orientable 
objects in $\iota_{>0}\mrE \fus$ has a preferred trivialization. The canonical $p_1$-orientation is the constant invariance 
datum. 
\item Given $X$, denote by $T':= \cT_X(S^1_{nb})$ the category for the circle with non-bounding $3$-framing. 
The constant invariance datum for~$X$ defines a braided tensor structure on $T\oplus T'$ and factorizations  
\[
GA_X = \mathrm{Aut}_{\mrE\fus}(X) \times G, \qquad
\cT_X(S^1_b)\oplus \cT_X(S^1_{nb}) \xrightarrow[braided]{\sim} \cT_X(S^1_b)\boxtimes \vect^{br, 1}_{\bZ/2} 
\] 
in which $\pi_1G=\bZ/2$ generates the second factor $\vect^{br,1}_{\bZ/2}$. 
\item A morphism $M\in \mathrm{Hom}_{\mrE \fus }(X; Y)$ between canonically $p_1$-oriented objects $X,Y$  
can be made compatibly $\SO_2$-invariant.
\item When $F\in \fus$ is a fusion category and $z\in GL_1 T$ with $z^2=\mathbf{1}$, the boundary 
theory for $\cT_F^z$ defined by an $F$-module~$M$ can be made $\SO_2$-invariant if and only if 
$z$ maps to $\mathbf{1}$ under $Z(F)\cong Z\left(\mathrm{End}_F(M)\right) \to \mathrm{End}_F(M)$. 
\item In (iv) above, the preferred root $\mu(F)^{1/4}$ of Theorem~\ref{thm5} agrees with 
$\beta_{z,z}\in \{\pm 1\}$, and accordingly, $c\left(\cT_F^z\right)= 0$ or $12\pmod{24}$. In the 
latter case, $\cT_F^z$ admits no $\SO_2$-invariant boundary conditions. 
\end{enumerate}
\end{maintheorem}

\begin{remark}
The core of Theorem~\ref{canonical} is Part~(iv). The proof is somewhat involved, and requires an 
elaboration of Theorem~\ref{thm5}, which we address momentarily in \S\ref{elab5}-\ref{amplify}. Before doing so, 
let us illustrate  the key check needed for relative $p_1$-orientability of $M$, in TQFT pictures (Figure~\ref{boundaryserre}). 
The bulk Serre functor~$S_F$ (blue embedded surface) can terminate on the boundary~$M$ as 
the relative Serre functor (framing twist)~$S_M$. Bulk orientability also allows an internal 
termination of~$S_F$ on the red line. An isomorphism between the two equivalences around the 
base square defines a termination (black dot) of the bulk ending of~$S_F$ (red) on the boundary~$M$, 
ending (and therefore trivializing)~$S_M$.

\begin{figure}[h]
\centering
\includegraphics[width=2.7in]{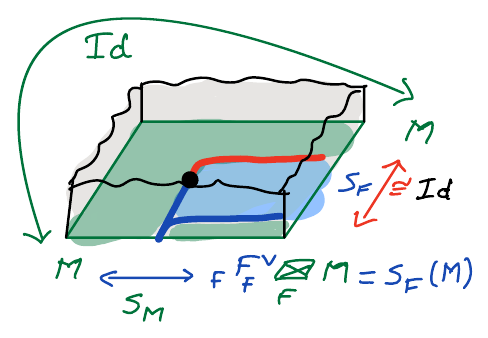} 
\caption{Trivializing the boundary Serre functor $S_M$}
\label{boundaryserre}
\end{figure}
\end{remark}
 
With respect to Part~(iv), when~$z$ maps to $\mathbf{1}\in F$, the resulting orientation on~$F$ and 
its regular boundary gives rise to a \emph{spherical structure} (see for instance \cite{dsps} for 
further discussion):

\begin{theorem}\label{spherical}
Let $F$ be a simple fusion category. The following are equivalent: \begin{itemize}
\item A spherical structure on~$F$;
\item An $\SO_3$-invariance structure on~$F$, compatible with some  $\SO_2$-invariance data 
on its regular module.
\end{itemize} 
\end{theorem}
\noindent
This meshes well with the main theorem of~\cite{ft1}: slightly loosely put, a simple fusion 
category is equivalent to a pair consisting of a simple TQFT $\cT_F$ and a non-zero boundary condition.

\subsection{The extended fusion category $T\oplus T'$.} \label{elab5}
Elements of $[\mathrm{Aut}^{br} T]$ represent isomorphism classes of invertible $T$-modules~\cite
{davnik}. The Serre functor~$S_X$ corresponds to the $T$-module $T' =\cT_X(S^1_{nb})$. The homomorphism 
$\Omega G \to \mathrm{Aut}_{\mrE\fus}(X)$ can be factored through the base~$\bZ/2=\pi_0\Omega G$, 
after decoupling the two Pontrjagin layers as explained following \eqref{unique}. This 
defines a \emph{$\bZ/2$-crossed braided structure} \cite{eno} on the category
\begin{equation}\label{spincat}
T\oplus T' = \cT_X(S^1_b)\oplus \cT_X(S^1_{nb}).
\end{equation}
The $E_2$ group~$\Omega GA_X$ of the previous section is a $K(\bZ/4;2)$-bundle 
over~$\mathrm{GL}_1(T\oplus T')$, 
pulled back from the grading $\Omega GA_X\to \bZ/2$. For instance, when~$\pi_1G$ injects into 
$[\mathrm{Aut}^{br}T]$, the factor~$T'$ is not isomorphic to $T$ as a module and contains no 
invertible objects; we only see $GL_1T$. 

\subsection{Braided tensor structures on $T\oplus T'$.} \label{amplify}
When~$T'\cong T$, \emph{and} if we can split the extension~\eqref{groupext3}, the homomorphism 
$\Omega G \to \mathrm{Aut}_{\mrE\fus}(X)$ may be retracted to the top layer~$B^2\bC^\times$ 
of the codomain. The retracted homomorphism deloops, in four possible ways, to a homomorphism 
$G\to B^3\bC^\times$, matching the choices of~$\mu(X)^{1/4}$ involved in splitting~\eqref 
{groupext4}. Each choice promotes~\eqref{spincat} to a $\bZ/2$-graded braided category~\cite
{davnik}. (This is because $H$-graded braided categories with an abelian~$H$ and 
identity component~$T$ are classified by $E_2$ homomorphisms $H\to \mathrm{Aut}(\mathrm{Id}_T)$, 
and the automorphism group of~$\mathrm{Id}_T$ is $B^2\bC^\times$.) Each outcome 
factors as follows, with the four options for $\zeta^4=1$:
\begin{equation}\label{factorTT}
T\oplus T' \cong T\boxtimes \vect^{br,\zeta}_{\bZ/2}. 
\end{equation}
Each option corresponds to a splitting section of~\eqref{groupext2}. Shearing a given splitting 
by a suitable order-$2$ element of~$\mathrm{GL}_1T$ (one that is even unique up to isomorphism) factors $GA_X$ 
as a product, and $\zeta$ is then the self-braiding of the lifted generator of~$\bZ/2$. 

\begin{definition}\label{defcan}
The \emph{canonical modular structure} on~$T$ and \emph{$p_1$-orientation} of~$X$ are defined by 
factoring $GA_X\subset T\oplus T'$ as in~\eqref {factorTT}, with $\zeta=1$. There results the \emph
{canonical lift}~$\kappa\in T'$ of the generator of~$\bZ/2$. For a $2$-torsion element 
$z\in \mathrm{GL}_1T$, we denote by $\cT_X^z$ the $p_1$-oriented TQFT defined from the sheared 
splitting by $z\kappa$. 
\end{definition}

\begin{remark}
For a general choice of~$z$, $\zeta=\beta_{z,z}$.
\end{remark}

\begin{proof}[Proof of Theorem~\ref{canonical}.i, ii] 
The framing-preserving rotation of $S^1_{nb}$ defines an automorphism~$\theta'$ of the 
identity of~$T'$. This $\theta'$ is a quadratic refinement of the 
braiding $\beta$ of $T$, in that
\[
\theta'(x\otimes y \otimes z) \theta'(x\otimes z)^{-1} \theta'(y\otimes z)^{-1} \theta'(z) = 
\beta_{y,x}\circ\beta_{x,y} \otimes \mathrm{Id}_z, \quad
	\forall x,y\in T, z\in T'
\]
as follows from the lantern relation for Dehn twists. A trivialization of~$S_X$ is a 
choice of~$T$-module isomorphism~$T\cong  T'$, identifying the two $\Spin$ circles, and is effected 
by an invertible object $b\in T'$. This transports~$\theta'$ to a balancing  
on $T$, 
\[
b^*\theta': a\mapsto \theta'(a\otimes b) 
	\circ\left(\mathrm{Id}_a\otimes\theta'(b)^{-1}\right). 
\]
The $\theta'(b)$-correction is to ensure that $b^*\theta' (\mathbf{1}) =1$. Comparing with the 
definition of $\rho$ (end of the proof of Theorem~\ref{thm5}), this is seen to be 
a ribbon iff $b^2\cong\mathbf{1}$. Changing $b$ by a $2$-torsion element~$t\in T$ changes the 
double-braiding action of~$\mathrm{db}(b)$ on~$T$ by~$\mathrm{db}(t)$. Non-degeneracy of~$\beta$ 
on~$T$ ensures that \emph{exactly one} choice~$\kappa$ for~$b$ is central in $T\oplus T'$. The 
resulting canonical structure splits~$\pi_1G$ in $GA_X \subset T\oplus T'$. Any other splitting 
differs from it by some order-two element $z\in GL_1T$.
\end{proof}

\begin{proof}[Proof of Theorem~\ref{canonical}.iii]
Assume that $X,Y$ are simple, and use the folding trick to pass to the fusion category 
$F:=X\boxtimes Y^\vee$, for which $M$ becomes a module: we are then reduced to Part~(iv).
\end{proof}

\begin{proof}[Proof of Theorem~\ref{canonical}.iv]
Our~$M$ is a point in the space $\iota_{>0}\mathrm{Hom}(\vect^\otimes;F)$, which gets promoted to 
the fiber of a bundle over $B\SO^{p_1}_2$ by the $\SO^{p_1}_2$-invariance of~$F$. The $2$-skeleton~$S^2\to 
B\SO^{p_1}_2$ gives an integral cohomology isomorphism through~$H^4$, so the obstruction to a section 
through~$M$ over~$B\SO^{p_1}_2$ is detected over~$S^2$, where we meet the action of~$\pi_1\Omega S^2=\bZ$ 
via the Serre functor~$S_M$ relative to~$F$. (There is no possible extension obstruction as in~\eqref
{groupext3}, since $\pi_1=\bZ$ is free.)   
If~$S_M$ is trivializable compatibly with the isomorphism~$S_F\cong\mathrm{Id}_F$, 
the categorical cutoff of~$M$ precludes a $p_1$-dependence in~$F$, forcing the compatible
$\SO_{3,2}$-invariance of the pair~$(F,M)$.

To trivialize~$S_M$, we may assume that~$F$ and~$M$ are simple. Replacing $F$ by the isomorphic 
object $\mathrm{End}_F(M)\in \fus$ reduces us to the case of the regular module $M ={}_FF$. The Serre 
functor~$S_F$ is identifiable with the double dual $**$ on the fusion category~$F$ (see \cite{dsps}). 
This also identifies the relative Serre functor~$S_M$ with the double dual. A tensor isomorphism 
$\mathrm{Id}_F \cong **$ would give the sought-after  compatible trivialization of $S_M$. 

Now, a trivialization of~$S_F\in \mathrm{Aut}_\fus(F)$ is a tensor isomorphism, for some fixed $u\in F$,
\begin{equation}\label{twistedserre}
\mathrm{Id}_F \cong \large\{ x\mapsto u\otimes  x^{**}
\otimes u^{-1}\large\},
\end{equation}
and a compatible trivialization of $S_M$ is a functorial isomorphism $\{m\xrightarrow{\ \sim\ } u\otimes m^{**}\}$ 
on the~$F$-module~$M$. This exists iff $u=\mathbf{1}\in F$, in which case the trivializations form a 
torsor over the automorphisms $\bC^\times$ of the identity functor of~$M$ (over $F$).

We show first that $u=\mathbf{1}$ for the canonical orientation. To do so, we shift the problem to 
the center~$Z(F)$ by ``squaring'' the pair $(F,M)$ to the category $F\boxtimes F^{op}$ and its regular 
module. This leads to the element $u\boxtimes u^{op}$ in the analogue of \eqref{twistedserre}. A compatible 
boundary orientation exists iff $u\boxtimes u^{op} = \mathbf{1}$, which in turn happens iff $u=\mathbf{1}$  
(since the diagonal map is injective on invertible isomorphism classes in~$F$). 
We will find the compatible orientation by replacing the pair $\left(F\boxtimes F^{op}, F\boxtimes F\right)$ 
in~$\fus$ with the isomorphic pair $\left(Z(F),F\right)$. 

Choosing $\mathbf{1}\in F$ as generating object over $Z(F)$, the module~$F$ gets identified with the 
category of (right) module objects, internal to~$Z(F)$, over the commutative algebra object 
$W:= \iota^* (\mathbf{1})$, where $\iota:Z(F)\to F$ is the natural map \cite[Prop.~8.8.8]{egno}. 
The functor~$\mathrm{Id}_F$ corresponds to the $W\text{-}W$-bimodule~$W$, and~$S_F$ to~$W^{**}$. 
On~$Z(F)$, the bulk Serre functor is identified with the identity by the factorization~\eqref{factorTT}: 
the canonical element~$\kappa$ identifies the bimodule~$T'$ with the identity bimodule~$T$. 
Relative orientability of~$F$ then amounts to the agreement of the Morita equivalence of algebra 
objects $W^{**}\equiv W$, mediated by~$W^*$ in~$Z(F)$ (the bottom edge in Figure~\ref{boundaryserre}),  
with the identification~$W^{**}\cong W$ induced by the canonical orientation of~$Z(F)$ (the right edge). 

Left-dualizing once, this datum is equivalent to an isomorphism of $W\text{-}W$ bimodules 
\[
W \cong (W^* = {}^*W)
\] 
with the natural left-and-right actions of~$W$ on the parenthesized objects and their identification via 
the canonical $** = \mathrm{Id}_{Z(F)}$. 
The obvious maps out of~$W$,
\[
W\cdot \mathbf{1}^* \to W^*,\qquad {}^*\mathbf{1}\cdot W\to {}^*W, 
\] 
are isomorphisms, being non-zero maps between invertible bi-modules. They are adjoint to the 
composition $W\otimes W \to W \to \mathbf{1}$, by transposing the left, respectively right factor. 
The (commutative) multiplication is the first step: this allows us to identify the two via the 
Drinfeld isomorphism in~$Z(F)$. This differs from the isomorphism induced from the canonical orientation 
by the balancing~$\theta=\kappa^*\theta'$, pulled back canonically from~$\theta'$ on~$T'$ (Definition~\ref
{defcan}). 

We claim that $\theta\equiv 1$ on $W$. The geometric proof is contained in Figure~\ref{Winvariance}: 
the object $\kappa\cdot W\in T'$ is the $\cT_F$-output of the annulus~$A$, with its product 
framing and input boundary colored by~$F$. The rotation of~$A$ is implemented in time, by the 
cylinder $A\times [0,1]$. However, the $3$-framing of this manifold is constant in time, so 
the picture is framed-diffeomorphic to the identity.

\begin{figure}[h]
\centering
\includegraphics[width=2.7in]{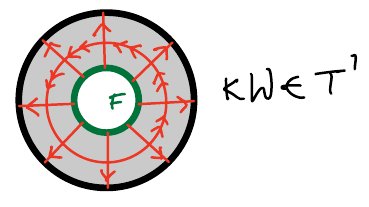} 
\caption{We have $\theta'\equiv 1$ on $\kappa\cdot W$ because of strict rotation-invariance} 
\label{Winvariance}
\end{figure}

This settles the canonical orientation. For a general bulk orientation~$z$,  the isomorphism 
defining our orientation structure,
\[
z\otimes:\mathrm{Hom}_{F\text{-}F}(F;F) \to 
	\mathrm{Hom}_{F\text{-}F}(F;F^\vee), 
\]
changes the identification $W\cong W^{**}$ by $z$-conjugation. We then need to identify the $W\text{-}W$ 
bimodules $W$ and $zW$. This can be done precisely iff $\iota(z)=\mathbf{1}\in F$.
\end{proof}

\begin{proof}[Proof of Theorem~\ref{canonical}.v] 
We know that $\mu(F)=1$ (Theorem~\ref{mudef}). We cannot have $\beta_{z,z}=\pm i$ in the center of~$F$ 
($z$ acts centrally on one of the fusion subcategories~$\vect^\otimes$ or $\vect\oplus \vect z$ of $F$), 
and the statement reduces to the observation following Definition~\ref{defcan}: the value of~$\mu(F)^{1/4}$ 
must match the self-braiding $\pm1$ of $z$. Note also that $\beta_{z,z} =-1$ prevents~$z$ from mapping 
to $\mathbf{1}$ in any fusion category Morita equivalent to $F$.   
\end{proof}

\begin{proof}[Proof of Theorem~\ref{spherical}] 
Figure~\ref{spherfig} depicts the solid ball, to which we apply the theory~$\cT_F$, with regular 
boundary condition and a self-defect boundary loop labeled by an object~$x\in F$. This computes the 
left and right traces of $x$, and their equality for all~$x$ is the sphericality condition, 
\cite[Def.~4.7.14]{egno}. The orientability of~$F$ and its regular module  ensures the 
isomorphism of the two pictures by rotating one of the spheres.

\begin{figure}[h]
\centering
\includegraphics[width=2.7in]{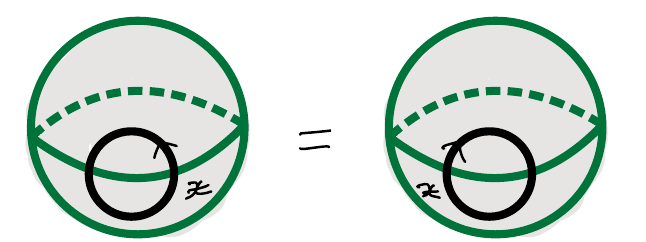} 
\caption{The unlabeled arm of the loop can be read as either~${}^* x$ or~$x^*$.} 
\label{spherfig}
\end{figure}

Conversely, as shown by Turaev \cite{tur} and M\"uger \cite{mug2}, a spherical structure on~$F$ leads 
to a modular structure on $Z(F)$. The canonical orientation then gives~$F$ a preferred pivotal structure. 
Any other pivotal structure differs from it by a tensor automorphism of $\mathrm{Id}_F$: this is realized 
by the double-braiding~$W(z)$ of \eqref{dblbraid} with a unique invertible $z\in \iota^{-1} 
(\mathbf{1})\in Z(F)$. When relating spherical structures, equality of left and right traces must 
be preserved, and this forces~$z$ to have order~$2$. The original spherical structure must then come from 
one of our changes from the canonical $p_1$-orientation.    
\end{proof}

\subsection{Four $p_1$-orientations of $\bZ/2$ gauge theory.}
\label{funnyZ2} 
The fusion category $\vect^\otimes_{\bZ/2}$, with its standard associator, generates the gauge theory 
with group~$\bZ/2$. Here, we discuss the four $p_1$-orientations corresponding to the four simple 
central elements $\{\mathbf{1}, e, m, em \}$ of the center $Z$ of $\vect^\otimes_{\bZ/2}$, with~$m$ 
denoting the generator of~$\bZ/2$ and~$e$ its sign representation. Three of them are genuine orientations, 
while the fourth has central charge $12\bmod{24}$. 

The framed theory has two simple boundary conditions, Neumann and Dirichlet, and the four simple line 
operators from~$Z$. Line operators with $\theta=-1$ need normal framings (here, $\bmod{\:2}$). 
The four possible ribbons are listed in Table~\ref{ribbontable}.

\begin{table}[h]
\centering
\begin{tabular}{c|c|c|c|c|}
& $\mathbf{1}$ & $m$ & $e$ & $em$ \\
\hline
$\theta^{1}$ & $1$ & $1$ & $1$ & $-1$ \\
\hline
$\theta^m$ & $1$ & $1$ & $-1$ & $1$ \\
\hline
$\theta^e$ & $1$ & $-1$ & $1$ & $1$ \\
\hline
$\theta^{em}$ & $1$ & $-1$ & $-1$ & $-1$ \\
\hline
\end{tabular}
\caption{Ribbons on $Z$}
\label{ribbontable}
\end{table}

The four $\bZ/2$-gauge theories are defined by counting the \emph{classical fields} described below, 
summing in a categorical sense when appropriate. The first three are genuinely oriented theories,  
while the fourth one has a sign anomaly for the orientation. Here are their explicit descriptions:
\begin{enumerate}\itemsep0ex
\item $\theta^1$ is the usual bundle-counting theory, and is defined even for unoriented manifolds. Neither 
boundary condition, nor any of the line operators requires a $\Spin$ structure along its support.
\item $\theta^m$ is the twisted version where $\Spin$ structures replace double covers as 
classical fields. The Dirichlet boundary, but not the Neumann one, requires a boundary $\Spin$ structure. 
The line operators $e$ and $em$ require $\Spin$ structures along their support: to define the monodromy 
of a Wilson loop, we need a base $\Spin$ structure to compare with the $\Spin$ field.  
\item $\theta^e$ counts bundles with weight $w_2(M)\cup w_1(m)\in K\left(\bZ/2;3\right)$, with the 
classes~$w_2(M)$ of the manifold and $w_1(m)$ of the bundle. The Neumann condition (but not 
the Dirichlet one) now requires a $\Spin$ structure on $\partial M$, to trivialize this twist. The `t Hooft loop labeled 
by~$m$ needs a longitudinal $\Spin$ structure, as does $em$: the bundle is singular, so $w_1(m)$ is not defined along the loop, 
and $w_2(M)$ must be trivialized there. 
\item $\theta^{em}$ is the $\Spin$-counting theory \emph{relative to the $p_1$-structure} $\tau$, 
$\delta\tau = p_1$. A $\Spin$ structure is a $1$-cochain~$w(m)$ trivializing~$w_2$, so $w_2\cup w(m)$ 
trivializes $p_1 \bmod{2}$. The difference $(\tau-w_2(M)\cup w(m)) \in K\left(\bZ/2; 3\right)$ is 
our counting twist. This theory has no $SO(2)$-invariant boundary conditions. Its central charge is 
$12 \bmod{24}$, since a unit shift in $\tau$ changes the sign of the invariants.
\end{enumerate}

\subsection{Non-canonical orientations on $F$.} \label{changeorient}
These four versions of gauge theory can be grafted onto a theory $\cT_F$ to effect a change in orientation. 
Let $z\in Z(F)$ have order~$2$. We tensor~$F$ over $\langle \mathbf{1},z\rangle$ 
with the four gauge theories, but the precise operation depends on the image of~$z$ in~$F$.

\emph{When $z$ maps to $\mathbf{1}\in F$.} This corresponds to~$\theta^1$ or~$\theta^e$ above. 
Half-braiding with $z$ defines a grading $F=F_0\oplus F_1$ on~$F$, which realizes $F$ as the 
$\bZ/2$-gauging of~$F_0$. We can build~$\cT_F$ by gauging~$F_0$ in the variant structure $\theta^e$, 
adding the $w_2$-weighting when counting bundles.

\emph{When $z$ does not map to $\mathbf{1}\in F$.} This orientation twist has two possible 
outcomes, matching $\theta^m$ or $\theta^{em}$, depending on whether $\langle \mathbf{1},z\rangle$ 
is $\vect_{\bZ/2}^{br,1}$ or~$\vect_{\bZ/2}^{br,-1}= S\vect$. (The $\zeta=\pm i$ braidings cannot occur, 
because $\langle \mathbf{1},z\rangle$ acts centrally on the image subcategory of $F$.) We use $z$ to 
`couple the category $F$ to the $\Spin$ structure' on the manifold where $\cT_F$ is 
to be evaluated. The  TQFT picture realizes~$\cT_F$ as a boundary of the $4$D gerbe theory 
defined by $\langle \mathbf{1},z\rangle$, and builds a sandwich by placing the $\Spin$-counting theory 
on the opposite side. 

%%%%%%%%%%%%%%%%%%%%%% SECTION %%%%%%%%%%%%%%%%%%%%%%

\section{Complex $p_1$-structures and central charge} 
\label{Qp1}
We now introduce an enhanced tangential structure, a \emph{complex $p_1$-structure}, 
capable of seeing the lift of $\underline{c}$ to a complex number. 

\begin{definition}
A \emph{$\bC p_1$-structure} on a real vector bundle is a trivialization of its first complexified Pontrjagin class.
\end{definition}

\noindent
The stable group $\Spin^{\bC p_1}$ is a $B^2\bC$-extension of $\Spin$. Restricting dimensions, such as 
in $\Spin^{\bC p_1}_{2,3}$, leads to $\bC p_1$-tangential structures on manifolds. The low 
$\bC p_1$-bordism groups for~$\Spin$ and~$\SO$ in dimensions $2,3,\infty$ are described in Table~\ref
{mytable} of Appendix~\ref{tangents}; most relevant here are
\[\begin{split}
\pi_3 \Sigma^3 MT \Spin^{\bC p_1}_3 &\cong \pi_3 M\Spin^{\bC p_1}\oplus 
	\pi_3 \Sigma^3 MT\Spin_3 \cong \bC/48\bZ \oplus \bZ/2, \\
\pi_3 \Sigma^2 MT\Spin^{\bC p_1}_2 &\cong \pi_3 M\Spin^{\bC p_1} \oplus 
	\pi_3 \Sigma^2 MT\Spin_2 \cong \bC/48\bZ\oplus \bZ/4.
\end{split}\]
A unit shift of $p_1$-structure represents $1 \in \bC/48\bZ$, while 
the generator of $\pi_3^s$ maps to $(2,1)$.
\begin{remark}
The lower groups $\pi_{0,1,2}$ are those of their $0$-cell $\bS^0$, except that $\pi_2 \Sigma^2 MT
\Spin^{\bC p_1}_2$ contains an extra $\bZ$ summand. 
\end{remark}
Their relevance is captured by the following.

\begin{maintheorem}\label{logmu}
\begin{enumerate}\itemsep0ex
\item A choice $\lambda(X)$ of $\log\mu(X)$ extends $\cT_X$ to $\Spin^{\bC p_1}_3$-manifolds, in such a
way that a shift in $\bC p_1$-structure by $4zp_1\in \bC p_1$ scales the manifold invariants 
by $\exp\left(z\lambda(X)\right)$.  
\item A chiral $\Spin$ CFT with central charge $c>0$ can be a boundary theory for the $\Spin^{\bC p_1}_3$ 
theory $\cT_X$ only if $\lambda(X) = 2\pi i c/6$. 
\end{enumerate}
\end{maintheorem}
\noindent
While the first statement is self-explanatory, the second may need review: see \S\ref{chiralrecall} below. 
One standard argument for it uses the projective cocycle of the conformal blocks measured on the moduli of 
curves and its relation to the Virasoro algebra. We will give another argument, using just the circle, 
which we think is novel.

\begin{proof}[Proof of Theorem~\ref{logmu}.i] 
An extension of $\cT_X$ to $\Spin^{\bC p_1}_3$-manifolds is equivalent to a trivialization 
of that group action on $X$. Through dimension $3$, $\Spin^{\bC p_1}_3 = B^2(\bC/\bZ)$ 
(with the generator of~$\pi_3$ for integral normalization) 
and it acts via the Bockstein
\[
\Sigma^2H(\bC/\bZ) \xrightarrow{\ B\ } \Sigma^3H\bZ 
	\xrightarrow{\ \mu(X)\ } \Sigma^3H\bC^\times. 
\]
A lift $\lambda:\bZ \to\bC$ of $\mu$ trivializes the composite homomorphism $\mu\circ B$, with 
the stated transformation law under shift of structure.  
\end{proof}

\subsection{Coupling to $2$-dimensional boundary theories.}
\label{chiralrecall}
Chiral CFTs, if not topological themselves, have a non-zero \emph {central charge $c\in \bC$}, manifested 
as a $\bC^\times$-central extension in the action of the diffeomorphism group $\mathrm{Diff}(S^1)$ 
of the circle on the chiral spaces of states.  ($\Spin$ structures may also be needed.) General 
CFTs also carry a \emph{Weyl anomaly} under conformal scalings of the metric, leading to real rescalings 
of their partition functions; we will disregard this aspect, and focus our discussion on the topological 
content of the chiral CFT. This is captured by its coupling, as a boundary theory, to a $3$D TQFT, and 
governed by Segal's \emph{modular functor} axioms \cite{seg}. This requires matching the tangential 
structures. $\Spin$ structures require us to distinguish between the bounding and non-bounding circles~$S^1_b$ 
and~$S^1_n$. The relevant diffeomorphism groups are double covers of~$\mathrm{Diff}(S^1)$, connected and trivial, 
respectively. However, this will turn out to affect the conformal weights in~\eqref{vanest} below, and not the 
central charge; as we focus on the latter, we will ignore that distinction.

\begin{proof}[Proof of Theorem~\ref{logmu}.ii]
A $\bC$-central extension of $\mathrm{Diff}(S^1)$ defines a class in its smooth group cohomology with 
complex coefficients. This is an extension of the cohomology of the Lie algebra $\mathfrak{diff}(S^1)$ 
by the de~Rham cohomology of the classifying space:\footnote{This is a corner of the \emph{van Est} 
spectral sequence.}
\begin{equation}\label{vanest}
\bC = H^2\left(B\mathrm{Diff}(S^1);\bC\right)
	\to H^2\left(B\mathrm{Diff}(S^1);\cO\right) 
	\to H^2\left(\mathfrak{diff}(S^1); \bC\right) = \bC.
\end{equation}
Restriction to $\mathrm{SL}_2(\bR)$ splits the sequence, because
$H^2\left(\mathfrak{sl}_2(\bR);\bC\right)=0$. 
The first component is the \emph{conformal weight} of a 
projective representation of~$\mathrm{Diff}(S^1)$: it measures the failure of the $\mathfrak{sl}_2(\bR)$-enforced 
Lie algebra splitting to extend to the rigid rotation subgroup. The second~$\bC$ component is the \emph{central 
charge}, normalized so that  $\mathfrak{diff}(S^1)$ acts on the free chiral complex spinor field with~$c=1$.

The isomorphism $H^2\left(\mathfrak{diff}(S^1); \bC\right) =\bC$ is a special ($1$-dimensional) case 
of \emph{Gelfand-Fuks cohomology}. Recall from  \cite{bs} that, for a $k$-manifold $M$, that cohomology agrees 
with that of the space of sections of a fiber bundle $GF(M)\twoheadrightarrow M$; the fiber is the truncation of the 
restriction of the universal bundle $E\mathrm{U}_k$ to the Schubert cells of $B\mathrm{U}_k$ 
of complex dimension $\le k$, and $GF(M)$ is associated to the frame bundle of $M$ by the inclusion 
$\mathrm{O}_k \to \mathrm{U}_k$.\footnote{This forcibly identifies  the Chern classes of 
$T_{\bC}M$ with the universal ones over the truncated $B\mathrm{U}_k$.} 

The complexified algebra $\mathfrak{diff}_\bC(S^1)$  is the maximal symmetry applicable to the conformal 
boundary, when a circle is embedded in a conformal germ. Crossing with an interval, a $2$-dimensional 
Gelfand-Fuks symmetry applies to the smooth germ around the circle, at the topological face. However, 
the Virasoro extension in~\eqref{vanest} \emph{does not} restrict from the $2$-dimensional 
Gelfand-Fuks cohomology over the circle: the relevant group is zero, as seen in the model below. 
This seems to prevent the coupling to oriented TQFTs, which accommodates the larger vector field 
symmetries. 

The resolution lies in the $\bC p_1$ restriction of tangential structure. This  
can be differentiably incorporated in the Gelfand-Fuks complex, by adding a \emph{Pontrjagin structure} 
$3$-cochain~$\pi$ to kill the complexified~$p_1$. The redundant annihilation by $\pi +c_1v_1 -2v_2$ 
leaves a surviving $3$-cocycle, whose transgression over~$S^1$ restricts to the Virasoro central extension. 
The cochains $\pi, v_2, c_1^2, c_2$ vanish upon $1$-dimensional restriction, while $c_1v_1$ transgresses 
to a multiple of the Virasoro cocycle. The local $2$-dimensional Gelfand-Fuks local models 
are in the tables below, where $v_{1,2}$ are the generators of $H^*\left(\mathrm{U}_2\right)$:

\begin{figure}[h]
\begin{minipage}{0.45\textwidth}
\centering
$
\xymatrix@1@=10pt{
v_2 \ar@/^/[dddrrrr]+<.6em,1.4ex> & 0 & c_1v_2 & 0 & c_1^2v_2, c_2v_2 \\
0 & 0 & 0 & 0 & 0 \\
v_1\ar[drr] & 0 & c_1v_1\ar[drr] & 0  & c_1^2v_1, c_2v_1\\
1 & 0 & c_1 & 0 & c_1^2,\quad c_2
}
$
\caption{Gelfand-Fuks model space: \newline
$d_2:v_1\mapsto c_1$,\hskip1cm $d_4:v_2\mapsto c_2$}
\label{gfspec1}
\end{minipage}
\begin{minipage}{0.45\textwidth}
\centering
$
\xymatrix@1@=10pt{
\pi,\quad v_2 \ar@{>}@<-1.5ex>[]+<0.4em,0ex>; [dddrrrr] +<0em,4.2ex>\ar@/^1pc/[dddrrrr]+<.6em,1.4ex> & 0 & c_1v_2 & 0 & c_1^2v_2, c_2v_2 \\
0 & 0 & 0 & 0 & 0 \\
v_1\ar[drr] & 0 & c_1v_1\ar[drr]+<-2em,1ex> & 0  & c_1^2v_1, c_2v_1\\
1 & 0 & c_1 & 0 & {\overbracket{c_1^2,\quad c_2}}
}
$
\caption{$\bC p_1$-structured model space: \newline
$d_2:v_1\mapsto c_1, \:\:\: d_4: \pi\mapsto 2c_2 -c_1^2, v_2\mapsto c_2$}

\label{gfspec2}
\end{minipage}
\end{figure}

Coupling the boundary CFT with a TQFT therefore requires matching the multiple of~$\pi$, the topological 
central charge measured in~$\bC p_1$-units, with the Virasoro central charge. This is the content 
of the Theorem, up to scale. Rather than matching cocycles, we can deduce the scale from the known 
$p_1$-variance~$c=1$ of the complex fermion theory. 
\end{proof}

\begin{remark}
In the same argument, general $rp_1$-structures instead of $\bC p_1$ lead to powers 
$\mathrm{Diff}^{\:rp_1}(S^1)$ of our central extension and to the equality $\lambda(X) = c \mod{24r}$.
\end{remark}

\begin{remark}
The condition in Theorem~\ref{logmu}.ii is not sufficient: the unstable summand~$\bZ/4$ 
from $\pi_3\Sigma^2MT\Spin_2$ also needs  matching. We do not know an instance where this 
problem occurs, although we can force a mismatch by tensoring the TQFT with an unstable 
invertible factor.
\end{remark}

%%%%%%%%%%%%%%%%%%%%%%%%% APPENDICES %%%%%%%%%%%%%%%%%%%

\appendix
\section{An exotic gauge theory}
\label{exgauge}
The new objects in $\mrE\fus$ allow the construction of additional $4$-dimensional TQFTs, which do not come 
from the fusion $2$-categories classified in~\cite{f2cat}. This is why we cannot quite rely on the latter 
to construct the one-step categorifications of $\mrE\fus$ and $\mrE S\fus$, the expected targets for 
$4$-dimensional TQFTs. We hope to return to this in a follow-up paper; here, we describe 
the exotic bosonic $4$-dimensional $\bZ/3$ gauge theory made possible by the group of 
units~$\langle U \rangle\subset\mrE \fus$ from~\S\ref {boson}. 

The object $A:=\vect^\otimes\oplus U^{\otimes 2} \oplus U^{\otimes 4} \in \mrE\fus$ has a canonical algebra structure: 
the vanishing group $H^4\big(B\bZ/3; \bC^\times\big)$ precludes any variation of the (higher) 
associator. It generates the $4$-dimensional gauge theory $\cG_{\bZ/3;\delta}$ for the group~$\bZ/3$ 
with the novel Dijkgraaf-Witten twist $\delta$, valued in the (base of the) group of units 
$\mathrm{GL}_1\,\mrE\fus$. The classifying space of the latter has $\pi_4=\bC^\times, \pi_1 =\bZ/6$ 
and $k$-invariant $Sq^2\times P^1_3$; this vanishes at this stage of delooping, and the map 
\[
\delta: B\bZ/3 \to B \mathrm{GL}_1\,\mrE\fus
\] 
is the unique lift of the inclusion map to the base (in light of the~$H^4$ vanishing mentioned.)

It is easy to compute the first Drinfeld center $\cG_{\bZ/3;\delta}(S^1_b)$ as an object in 
$\mrE\fus$: the answer is  
\begin{equation}\label{exotics1}
\left(\vect^\otimes\oplus U^{\otimes 2} \oplus U^{\otimes 4}\right)
	\boxtimes \vect_{\bZ/3},
\end{equation}
the original~$A$ tensored with the generator of $3$-dimensional gauge theory. 

We can now evaluate the reduced theory on $3$-spheres with various framings. The $\vect_{\bZ/3}$ 
factor gives a factor of~$1/3$, from bundle automorphisms, but the $U$-powers in the first factor 
lead to framing-dependent answers: we get $3$, for framings representing multiples of~$3$ in the 
bordism group, and $0$ otherwise. This matches the sum of integrals, with coefficients 
in~$\Sigma^3\bI_{\bC^\times}$, of the transgression of~$\delta$ over~$S^1_b$. 

This TQFT does not match the standard~$\bZ/3$-gauge theory, generated by the algebra object 
$\vect^\otimes\oplus \vect^\otimes\oplus \vect^\otimes \in \fus$ with the $\bZ/3$-convolution, as 
the latter is insensitive to framing; nor can it be built from any other algebra object in $\fus$. 
Indeed, from~\eqref{exotics1}, we compute the M\"uger center $\cT_{\bZ/3;\delta}(S^2)$ to 
be~$\mathrm{Rep}(\bZ/3)$, so any candidate must be some version of $\bZ/3$-gauge theory. 
However, there are no Dijkgraaf-Witten twists that we could use  without the units~$U$. 

%%%%%%%%%%%%%%%%%%%%%%%%%%%% SECTION %%%%%%%%%%%%%%%%

\section{Tangential structures involving $p_1$}
\label{tangents}

We review the $p_1$-tangential structures on $3$-manifolds and on boundary surfaces, where 
TQFTs can interface with Conformal Field Theories. The main results are listed in Table~\ref{mytable} 
and the associated Study Guide, and the proofs are given in the text that follows.  

\subsection{Variations.} 
Framing a $3$-manifold defines a $\Spin$ structure. This is preserved by a \emph{local} change 
of framing (one concentrated near a point), because based maps $S^3\to \SO_3$ lift uniquely to $\Spin_3$. 
A $\Spin$ structure on manifold also defines a $3$-framing up to local change, resulting in the close 
relationship between linear TQFTs on framed manifolds\footnote {This does not apply to \emph{families 
of manifolds}, because $\pi_4 \Spin_3$ and the higher homotopy groups do not vanish.}  and anomalous 
TQFTs on $\Spin$ manifolds reviewed in \S\ref{spinomaly}. Note that \emph{$3$-framings} 
and \emph{stable framings} are distinct structures, 
because the inclusion $\SO_3 \subset \SO$ has index~$2$ on~$\pi_3$: there are twice as many 
local changes of stable framing. Extending a framed theory to stably framed manifolds is always 
possible, but meets an ambiguity of order~$2$ (Theorem~\ref {thm4}). Other instances where 
factors of $2$ cause problems, if not tracked correctly, relate $\Spin_3$-structures, $\Spin$ structures 
and orientations. Here, we bring some order with an explicit descriptions of bordism groups 
and maps between them.

\subsection{Tangential structures.}
\label{bordgroups}
An \emph{$n$-dimensional tangential structure $\tau$} is a space  equipped  with a map to~$B\mathrm{O}_n$, 
and a $\tau$-structure on a manifold is a factorization via $\tau$ of the structure map of the tangent bundle. 
Commonly, this represents a reduction of structure group, such as to the trivial group (an $n$-framing), 
the groups $\SO_n,\Spin_n \to \mathrm{O}_n$, but can be defined more generally. Thus, the homotopy 
fiber of the map representing the Pontrjagin class 
\[
p_1: B\SO_n \to \Sigma^4H\bZ
\]
defines an oriented $p_1$-structure, denoted by $\SO^{p_1}_n$. $\Spin^{p_1}_n$-structures 
are defined similarly. \emph{Signature structures}, considered in~\cite{tur} for oriented RT theories, are 
more subtle: see~\cite{projsym}. 

The bordism group of $n$-manifolds with $\tau$-structure is~$\pi_0$ of the respective 
\emph{Madsen-Tillmann spectrum} $MT(\tau)$, the (unreduced) de-suspension of the underlying space 
of~$\tau$ by the standard representation of $\mathrm{O}_n$, \cite{gmtw}. This is the ($n$-fold 
looped) group completion of the respective \emph{bordism category}. The~$MT$ spectrum for $n$-framings is the 
shifted sphere $\bS^{-n}$. 

\subsection{Bordism groups.}
The definitional shift in~$MT_n$ explains the (unnatural) 
indexing of homotopy groups in 
the following table of $3$-dimensional bordism groups: 

\begin{table}[h]
\centering
\def\arraystretch{1.3}
\begin{tabular}{|c|c|c|}
\hline
$\pi_1MT\Spin^{p_1/4}_2 = \pi_3^s\oplus \pi_1^s$&
$\pi_0 MT\Spin^{p_1/4}_3 = \pi_3^s$& 
N/A
\\
\hline
$\pi_1MT\Spin^{p_1/2}_2 = \pi_3^s\oplus \bZ/4$ & 
$\pi_0 MT\Spin^{p_1/2}_3 = \pi_3^s\oplus \bZ/2$ &
$\pi_3 M\Spin^{p_1/2} = \pi_3^s$
\\
\hline
$\pi_1 MT\Spin^{p_1}_2 =\bZ/48\oplus \bZ/4$ &
$\pi_0 MT\Spin^{p_1}_3 = \bZ/48\oplus \bZ/2$ & 
$\pi_3 M\Spin^{p_1} = \bZ/48$
\\
\hline
$\pi_1 MT\Spin^{\bC p_1}_2 = \bC/48\bZ\oplus \bZ/4$ & 
$\pi_0MT\Spin^{\bC p_1}_3 = \bC/48\bZ\oplus\bZ/2$ & 
$\pi_3M\Spin^{\bC p_1} = \bC/48\bZ$
\\
\hline
$\pi_1MT\Spin_2 =\bZ/4$ &
$\pi_0MT\Spin_3 =\bZ/2$ & 
$\pi_3 M\Spin = 0$
\\
\hline
$\pi_1 MT\SO^{p_1}_2 =\bZ/12$ &
$\pi_0 MT\SO^{p_1}_3 =\bZ/6$& 
$\pi_3 M\SO^{p_1} = \bZ/3 $ 
\\
\hline
$\pi_1MT\SO^{\bC p_1}_2 = \bC/12\bZ$ &
$\pi_0MT\SO^{\bC p_1}_3 = \bC/6\bZ$ &
$\pi_3M\SO^{\bC p_1} = \bC/3\bZ$
\\
\hline
\end{tabular}
\caption{Some relevant bordism groups}
\label{mytable}
\end{table}
\subsection{Study Guide.} 
Before describing the most important groups more explicitly and the maps 
between them  in Proposition~\ref{spinmaps}, here are some quick 
pointers:
\begin{enumerate}\itemsep0ex
\item There are natural maps going down and to the right.
\item The horizontal maps are surjective and compatible with the splittings indicated. 
\item For $r\in\{1/2,1,\bC\}$, the $\Spin^{rp_1}_n$-structured groups ($n=2,3$) may be described  uniformly:  
\begin{equation}\label{splitspin3}
\pi_3\Sigma^n MT\Spin^{rp_1}_n = 
	\pi_3 M\Spin^{rp_1} \oplus \pi_3 \Sigma^nMT\Spin_n; 
\end{equation} 
however, the image of the generator of $\pi_3^s$ has $1$ in its second component. 
\item The similar splitting for the $\SO^{p_1}$ groups is misleading: it is not compatible   
with the maps from $\Spin$. In fact, all $\SO^{p_1}$-groups are quotients of $\pi_3^s$, from 
the lowest cell. 
\item As a consequence, in the last line  we find extensions of~$\bC/3\bZ$ instead of splittings.
\item The group $\pi_3^s$ in $M\Spin^{p_1/2}$ comes from the inclusion of $\bS^0$. The presentation 
$\pi_3^s=2\bZ/48$ makes the vertical inclusions in the first four rows natural. 
\item Shifts of $rp_1$-structure\footnote{On $\Spin^{p_1/2}$-manifolds, only even 
shifts of $p_1$-structure are executable, since one `tick' of structure shifts~$p_1$ by~$2$.} 
cycle through the stable summands in the $\Spin$ groups, and through the full $\SO$ groups.
\end{enumerate}

\subsection{Framings and stable framings.} 
On $3$-manifolds, $\Spin^{p_1/4}_3$-structures are the same as~$3$-framings.
Rigidifying the bordism $3$-category above dimension~$3$ by collapsing diffeomorphism groups of 
$3$-manifolds to their component groups identifies the respective structures.    
Now, the map $\pi_3\SO_3\to \pi_3\SO$ has index~$2$. Therefore, stabilizing the $3$-framing on a 
$3$-manifold has the same effect as first killing~$w_2$, and then killing~$p_1/2$ on $\SO_3$. The 
previous argument identifies stably framed and $\Spin^{p_1/2}_3$-structures on the (rigidified) 
bordism $3$-category.

\subsection{Stable groups.} The $4$-cell in $M\Spin$ is attached to the $0$-cell via the generator 
of~$\pi_3^s$: we know this from the vanishing of the bordism group $\pi_3M\Spin$. Supplying $p_1/2$ 
structure kills the $4$-cell, so that $\pi^s_3 \to \pi_3M\Spin^{p_1/2}$ is an isomorphism. More generally, 
the spectra $M\Spin^{rp_1}$ are  extensions (normalized by $p_1$)
\begin{equation}\label{mspinp1}
\begin{aligned}
\bS^0 \rightarrowtail &M\Spin^{rp_1} \twoheadrightarrow \Sigma^3H(\bZ/2r\bZ) \\
\bS^0 \rightarrowtail &M\Spin^{\bC p_1} \twoheadrightarrow \Sigma^3H(\bC/2\bZ) 
\end{aligned}
\qquad (\text{through } \pi_5),
\end{equation}
classified by Bocksteins maps to $\Sigma^4\bZ$ followed by its surjection to the generator of $\pi_4\bS^1$. 
This leads to the extended, topologically cyclic~$\pi_3$ groups in the right 
column of Table~\ref{mytable}. 

\begin{remark}
The cokernel $\bZ/2$ of the map $\pi_3^s\to 
M\Spin^{p_1}$ is detected by the difference between the mod~$2$ reduction of the $p_1$-structure 
and~$Sq^2$ of the $\Spin$ structure, viewed as trivialization of~$w_2$. (The identity $Sq^2w_2 = 
p_1\bmod{2}$ is respected on trivializations induced by a framing, but not for independent 
trivializations.)
\end{remark}

The stable groups~$\SO^{p_1}, \SO^{\bC p_1}$ are determined from the identification 
$p_1/3: \pi_4M\SO\xrightarrow{\ \sim\ } \bZ$, the vanishing of lower homotopy groups and the 
Bockstein argument from before.   

\subsection{Dimension $3$.}
The isomorphism $\pi_0MT\Spin_3\cong \bZ/2$ can be seen from the exact sequences
\[
\xymatrix{
0\ar[r] & \pi_4MT\Spin_3 \ar[r]^-{12}\ar[d] & \pi_4\bS^4 \ar[r]^2\ar[d]^-{2} & \pi_3^s \ar[r]\ar[d]^-{1} 
	& \pi_3MT\Spin_3 =\bZ/2 \ar[r]\ar[d] & 0 \\
0 \ar[r]& \pi_4M\Spin \ar[r]^-{24}& \pi_4\bS^4 \ar@{>>}[r]& \pi_3^s \ar[r]& \pi_3M\Spin =0 \ar[r]& 0 
}
\]
which also identify $\pi_4 MT\Spin_3 = \pi_4 M\Spin = 48\bZ$, in units of the dual homology 
class~$p_1^\vee$. 

The  spectra $\Sigma^3MT\Spin_3^{rp_1}$, for $r=1,\bC$, are extensions through $\pi_5$, as 
in~\eqref{mspinp1}; but their classifying maps now have index~$2$, as seen from the top sequence  
in the diagram above. This leads to the~$\bZ/2$ summands in middle column of Table~\ref{mytable}. 

\subsection{Dimension $2$.} 
The low homotopy of the spaces $MT\Spin_2^{rp_1}$ can be determined from the fibration, 
analogous to~\eqref{mspinp1}, 
\[
\Sigma^{-TS^2}S^2 = \bS^{-2}\oplus \bS^0 \rightarrowtail  MT\Spin_2 \twoheadrightarrow  \bS^2 
\qquad\text{(below } \pi_3)
\] 
with the source split because $TS^2$ is stably trivial. A~$p_1/4$-structure kills the $4$-cell 
in $B\Spin_2$; this leaves $MT\Spin^{p_1/4}_2 \cong \bS^{-2}\oplus\bS$ 
up to $\pi_1$. 

A $p_1$-structure on a family oriented surfaces trivializes the $12$th power of the Hodge 
determinant bundle $\delta:= \det\left(H^*(\cO)\right)$. The~$12$ (now flat) powers of~$\delta$ 
detect the group $\pi_1MT\SO^{p_1}_2$, and shifts in $p_1$-structure cycle through it. 

$\Spin$ structures allow us to define fourth roots of $\delta$, from the Pfaffian of the Dirac 
operator; the four roots differ by flat line bundles, which detect the group $\pi_1 MT\Spin_2\cong  \bZ/4 $. 
This additional divisibility also enhances the $\bZ/12$ of 
$\pi_1MT\SO^{p_1}_2$ to the $\bZ/48$ 
summand of $\pi_1MT\Spin^{p_1}_2$, in addition to the $\bZ/4$ from $ \pi_1MT\Spin_2$.

\begin{remark}
A proof of the splitting~\eqref{splitspin3} on~$\pi_3$ with $r\in\{1/2, 1,\bC\}$ can be seen directly the following 
extensions, stemming from the analogues of the fibration~\eqref{mspinp1} in dimensions~$2$ and $3$:
\begin{alignat*}{3}
\pi_3^s\oplus \pi_1^s &\rightarrowtail 
	\pi_1 MT\Spin^{rp_1}_2 &\twoheadrightarrow \bZ/4, \\
\pi_3^s &\rightarrowtail 
	\pi_0 MT\Spin^{rp_1}_3 &\twoheadrightarrow \bZ/4,\\
\pi_3^s &\rightarrowtail \quad
	\pi_3 M\Spin^{rp_1} &\twoheadrightarrow \bZ/2.
\end{alignat*}
All oriented $3$-manifolds admit $p_1$-structures, so the forgetful maps are surjective. Shifts 
in $p_1$-structure act cyclically on the stable summand, so the four, respectively two orbits in 
the left groups must map bijectively onto the right summands in~\eqref{splitspin3}, establishing 
the splitting.

The extensions are classified by the elements $(2,1), 2$ and $1$ in the respective $\mathrm{Ext}$ groups. 
\end{remark}

\subsection{Maps between structural groups.}
The homotopy groups align to $3$-manifold structures and fit into a natural diagram of maps induced by relaxing the structures:
\[ %\label{groupmaps}
\xymatrix{
 & \pi_1 MT\Spin^{p_1/2}_2 \ar[r] \ar[d] & \pi_1 MT\Spin^{p_1}_2\ar[r]^-{o_2} \ar[d] &
	\pi_1 MT\SO^{p_1}_2 \ar[d] \\
{\begin{matrix}\pi_3^s\cong \pi_0 MT\Spin^{p_1/4}_3 \end{matrix}}
	\ar[r]^-{s_3} \ar[ur]^-{s_2} 
	\ar[dr]^{s}_{\sim} 
	&
	\pi_0 MT\Spin^{p_1/2}_3\ar[r] \ar[d] &
	\pi_0 MT\Spin^{p_1}_3 \ar[r]^-{o_3}\ar[d] & 
	\pi_0 MT\SO^{p_1}_3 \ar[d],\\
 & \pi_3 M\Spin^{p_1/2}\ar[r] & 
	\pi_3 M\Spin^{p_1} \ar[r]^-{o} & 
	\pi_3 M\SO^{p_1}. 
}
\]
\begin{proposition}\label{spinmaps}
All vertical maps are surjective, the $s_\bullet$ are injective, while all $o_\bullet\circ s_\bullet$ are surjective. More precisely, we can choose generators so that 
\[
\xymatrix{
& \bZ/24 \oplus \bZ/4 \ar[r]^-{(2,1)}\ar[d]^- 
	{\left[
	\begin{smallmatrix} 1 & 0 \\ 0 & 1 \end{smallmatrix}\right]}& 
	\bZ/48\oplus \bZ/4 \ar[r]^-{(-1,\pm 3)} 
	\ar[d]^{\left[\begin{smallmatrix} 1 & 0 \\ 0 & 1 \end{smallmatrix}\right]}
	& \bZ/12 \ar[d]^{1} \\
\bZ/24 \ar[dr]^{=}
	\ar[r]^-{(1,1)} \ar[ur]^-{(1, \pm 1)}
	& \bZ/24 \oplus\bZ/2 \ar[r]^{\left[
	\begin{smallmatrix} 2 & 0 \\ 0 & 1 \end{smallmatrix}\right]} 
	\ar[d]^{\left[
	\begin{smallmatrix} 1 \\ 0 \end{smallmatrix}\right]} 
	& \bZ/48\oplus \bZ/2 \ar[r]^{\qquad(-1,3)}\ar[d]^
	{\left[
	\begin{smallmatrix} 1 \\ 0 \end{smallmatrix}\right]}
	& \bZ/6 \ar[d]^{1}\\
& \bZ/24 \ar[r]^{2}& \bZ/48 \ar[r]^{1} & \bZ/3
}
\]
The sign ambiguity in the top row is absorbed by the sign automorphism of $\bZ/4$.
\end{proposition}

\begin{proof}[Sketch of proof.]
The groups are determined and compared from the Atiyah-Hirzebruch spectral sequences for 
\emph{twisted} $I^*_{\bC^\times}$-cohomologies of the respective $B\Spin$ and $B\SO$ groups. The 
twistings are by the (negative of the) standard representation, acting on coefficients via the $J$-homomorphism.   

We can certainly choose generators of $\bZ/24, \bZ/48, \bZ/12, \bZ/6$ and $\bZ/3$, compatibly 
with the maps indicated. Surjectivity of $o_\bullet\circ s_\bullet$ shows that the maps 
out of the unstable kernels to the rightmost groups must be injective, and the ones out 
of $\pi_3^s$ to them must be surjective. 

Finally, a $p_1$-structure on surfaces is trivialization of the $12$th power of the Hodge determinant 
bundle $\delta$, and shifts cycle through the twelve (now flat) powers of $\delta$. 
The split $\bZ/48$ summand must then surject through $o_2$. Combined with the injectivity of $s_3$ 
on $\bZ/2$ and the choice of sending the standard generator $1\in\bZ/24$ to the standard ones 
in the rightmost groups, this pins the maps, up to the sign ambiguity flagged. 
\end{proof}

\subsection{The free fermion.}
\label{ffermion}
One consequence concerns the free fermion theory $\ferm$. The standard character of $\pi_3^s$ can be 
extended to $\pi_0MT\Spin^{p_1}_3$ in four different ways, as seen from Proposition~\ref{spinmaps}; 
we use the extension~\eqref{psidef} to define~$\psi$. This is the \emph{reflection-positive} choice, 
according to \cite{fh}. One observation, which can cause much grief in the form 
of apparent (although not genuine) contradictions,  is the following
\begin{proposition}
When restricted to framed manifolds, the theory $\ferm^{\otimes 4}$ factors uniquely through 
$\SO^{p_1}_3$-structures. 
Starting from manifolds with $\Spin^{p_1}_3$ structure, only powers of $\ferm^{\otimes 16}$ factor. 
Similarly, $\ferm^{\otimes 2}$ factors from framed through $\SO^{p_1}_2$-manifolds, but only 
powers of $\ferm^{\otimes 16}$ descend there from $\Spin^{p_1}_2$-structures. 
\end{proposition} 
\begin{proof}
This is clear from the explicit factoring maps in Proposition~\ref {spinmaps}: all kernels of the right 
horizontal arrows are isomorphic copies of $\bZ/16$.  
\end{proof}

%%%%%%%%%%%%%%%%%%%%%%%% BIBLIO %%%%%%%%%%%%%%%%%%

\bigskip
\small{
\noindent
\textsc{Daniel S.~Freed,} Harvard University, \texttt{dafr@math.harvard.edu} \\
\textsc{Claudia I.~Scheimbauer,} Technische Universit\"at M\"unchen, \texttt{scheimbauer@tum.de} \\
\textsc{Constantin Teleman,} UC Berkeley, \texttt{teleman@berkeley.edu} }

\end{document}